\documentclass[12pt]{amsart}

\usepackage{amsthm,amsmath,amsxtra,amscd,amssymb,color,esint,verbatim}
\numberwithin{equation}{section}

\newcommand{\ord}{\operatorname{ord}\nolimits}
\newcommand{\CC}{{\mathbb{C}}}

\newcommand{\PP}{{\mathbb{P}}}

\newcommand{\ZZ}{{\mathbb{Z}}}

\newcommand{\calA}{{\mathcal A}}
\newcommand{\calC}{{\mathcal C}}

\newcommand{\calM}{{\mathcal M}}

\newcommand{\bK}{{\textbf K}}
\newcommand{\bw}{{\boldsymbol\omega}}
\newcommand{\tK}{{\widetilde{K}}}
\newcommand{\tw}{{\widetilde{\omega}}}
\newcommand{\tbK}{{\widetilde{\textbf K}}}
\newcommand{\tbw}{{\widetilde{\boldsymbol\omega}}}
\newcommand{\txi}{{\tilde{\xi}}}
\newcommand{\teta}{{\tilde{\eta}}}

\newcommand{\hXi}{{\widehat{\Xi}}}
\newcommand{\hC}{{\widehat{C}}}
\newcommand{\tbC}{{\widetilde{C}}}

\newcommand{\op}{\operatorname}

\newcommand{\hol}{\op{hol}}
\newcommand{\res}{\op{res}}

\newcommand\us{{\underline{s}}}

\newcommand\tC{{\widetilde{C}}}

\newcommand\ut{{\underline{t}}}

\newcommand\uphi{{\underline{\phi}}}

\newcommand{\ugamma}{\underline{\gamma}}

\theoremstyle{plain}
\newtheorem{thm}{Theorem}[section]
\newtheorem{lm}[thm]{Lemma}
\newtheorem{prop}[thm]{Proposition}
\newtheorem{cor}[thm]{Corollary}

\theoremstyle{definition}
\newtheorem{df}[thm]{Definition}
\newtheorem{notat}[thm]{Notation}
\newtheorem{rmk}[thm]{Remark}

\begin{document}

\title{General Variational Formulas for Abelian Differentials}
\author{Xuntao Hu}
\address{Mathematics Department, Stony Brook University,
Stony Brook, NY 11794-3651, USA}
\email{xuntao.hu@stonybrook.edu}
\author{Chaya Norton}
\address{Concordia University, Montreal, QC, Canada, and Centre de Recherches Math\'ematiques (CRM), Universit\'e de Montr\,eal, QC, Canada}
\email{nortonch@crm.umontreal.ca}

\begin{abstract}

We use the jump problem technique developed in a recent paper \cite{GKN17} to compute the variational formula of any stable differential and its periods to arbitrary precision in plumbing coordinates. In particular, we give the explicit variational formula for the degeneration of the period matrix, easily reproving the results of Yamada \cite{Yam80} for nodal curves with one node and extending them to an arbitrary stable curve. Concrete examples are included.

We also apply the same technique to give an alternative proof of the sufficiency part of the theorem in \cite{BCGGM16} on the closures of strata of differentials with prescribed multiplicities of zeroes and poles.

\end{abstract}

\maketitle

\section{Introduction}\label{intro}


We work over the field of complex numbers $\CC$. Let $\calM_g$ be the moduli space of curves, and $\overline{\calM}_g$ be its Deligne-Mumford compactification. Given a stable nodal curve $C$ with $n$ nodes, the standard \textit{plumbing} construction cuts out neighborhoods of size $\sqrt{|s_e|}$ on the normalization of $C$ at the two pre-images $q_e$ and $q_{-e}$ of each node $q_{|e|}$ of $C$, and identifies their boundaries (called \textit{seams}, denoted by $\gamma_{\pm e}$) via a gluing map $I_e$ sending $z_e$ to $z_{-e}:=s_e/z_e^{-1}$, where $|s_e|\ll 1$ is called the {\em plumbing parameter} and $z_e$ and $z_{-e}$ are chosen local coordinates near $q_e$ and $q_{-e}$ respectively. The irreducible components of the nodal curve $C$ are denoted by $C_v$.  

The plumbing construction thus constructs a family of curves $\calC\to \Delta$ with the central fiber identified with $C$, where $\Delta$ is the small polydisc neighborhood of $0\in \CC^n$ with coordinates given by the plumbing parameters $\us:=(s_1,\ldots,s_n)$. Depending on circumstances $\us$ are also called the {\em plumbing coordinates}, as they give versal deformation coordinates on $\overline\calM_g$ to the boundary stratum containing the point $C$. The Riemann surface resulting from plumbing with parameters $\us$ is denoted by $C_\us$.

\subsection{Motivations} We are interested in studying the degeneration of the periods of an Abelian differential in plumbing coordinates, which can be seen as a direct application of the study of the behavior of degenerating families of Abelian differentials in plumbing coordinates. Our interest in the degeneration of the periods is motivated by the following two sources: the degeneration of period matrices near the boundary of the Schottky locus in $\overline\calA_g$; and the degeneration of the period coordinates (both absolute and relative periods) of strata in the Hodge bundle (also known as the moduli spaces of Abelian differentials with fixed multiplicities at marked points). In this paper we will only study the degeneration of period matrices. Application to period coordinates near the boundary of the strata will be treated elsewhere.


\subsection{Results and structure of the paper} \label{intro_results}
In this paper we give a complete answer to the question stated above. In order to state our results, we introduce some more notation. 

The Hodge bundle $\Omega\calM_g$ is a rank $g$ vector bundle over $\calM_g$. A point in the Hodge bundle corresponds to a pair $(C,\Omega)$ where $\Omega$ is an Abelian differential on $C$. One can extend the Hodge bundle over $\overline\calM_g$. The fiber of the extension $\Omega\overline\calM_g$ over a nodal curve $C$ in the boundary of $\overline\calM_g$ parametrizes {\em stable differentials}, that is, meromorphic differentials that have at worst simple poles at the nodes with opposite residues.

The goal in our paper is to compute the variational formula for any stable differential and its period over any $1$-cycle on $C$ in plumbing coordinates. The term ``variational formula" in our paper means an expansion in terms of both $s_e$ and $\ln{|s_e|}$. Note that a variational formula in this sense is {\em not} synonymous to a power series expansion in plumbing coordinates $\us$. We will use specifically the term ``$\us$-expansion" when we mean the latter, where no logarithmic terms are involved. Moreover, throughout the paper objects subscripted by ``$e$'' are indexed by the set of edges of the dual graph of the stable curve $C$, and those by ``$v$'' are indexed by the set of vertices of the dual graph. 

The technique we use to construct the degenerating family of Abelian differentials $\Omega_\us$ along the plumbing family $C_\us$ is called (solving) the jump problem, which will be properly defined in Section \ref{sec:notation}. The main idea is that given a stable differential $\Omega$ on $C$, we have the mis-matches $\{\Omega|_{\gamma_e}-I_e^*(\Omega|_{\gamma_{-e}})\}$ (which we call the \textit{jumps} of $\Omega$) on the seams $\gamma_{\pm e}$ at opposite sides of each node $q_{|e|}$. The solution to the jump problem is a ``correction" differential $\eta$ that matches the jumps of $\Omega$ with opposite sign. By adding $\eta$ to $\Omega$ on each irreducible component, one obtains new differentials with zero jumps, that can thus be glued to get a global holomorphic differential $\Omega_\us$ on $C_\us$. 

The jump problem is a special version of the classical Dirichlet problem. It was developed and used in the real-analytic setting in a recent paper by Grushevsky, Krichever and the second author \cite{GKN17}. In the classical approach, the Cauchy kernel on the plumbed surface is used, while in \cite{GKN17} the fixed Cauchy kernels on the irreducible components of the nodal curve are used, which is crucial to obtain an $L^2$-bound of the solution to the jump problem in plumbing parameters.

Our construction of the solution to the jump problem largely follows the method in that paper. Instead of the real normalization condition used in \cite{GKN17}, we normalized the solution by requiring that it has vanishing $A$-periods, where $A$ is a set of generators of a chosen Lagrangian subgroup of $H_1(C_\us,\ZZ)$ containing the classes of the seams. This normalization condition allows us to work in the holomorphic setting (as opposed to the real-analytic setting in \cite{GKN17}) where we can use Cauchy's integral formula. As a consequence, the construction of the solution is simpler.

In Section \ref{sec:mainresult} of our paper, the solution $\eta$ to the jump problem is constructed explicitly as an iterated series (see (\ref{mainjumps}) and (\ref{mainsol})). We show that such a solution depends analytically on the plumbing parameters and moreover, we bound the $L^2$-norm of $\eta$ by a power of the $L^\infty$-norm of $\us$. For convenience we denote $\hol(\Omega)$ the regular part in the Laurent expansion of the function $\frac{\Omega(z_e)}{dz_e}$ given in the local coordinate $z_e$ that was used to define plumbing near the node $q_e$.

\begin{thm}(=Theorem \ref{main})\label{thm_intro}
Let $(C,\Omega)\in\partial\Omega\overline\calM_g$ be a stable differential. For each $v$, let $\Omega_v$ be the restriction of $\Omega$ on $C_v$. There exists a unique solution $\{\eta_v=\eta_{v,\us}\}$ with vanishing $A$-periods to the jump problem for $\Omega$, such that for any $|\us|$ small enough and all $|s_e|>0$, $\{\Omega_{v,\us}:=\Omega_v+\eta_{v}\}$ defines a holomorphic differential $\Omega_\us$ on $C_\us$ satisfying $\Omega_v=\lim_{\us\to 0}\Omega_\us|_{C_v}$ uniformly on compact sets of $C_v\setminus\cup_{e\in E_v}q_e$. Furthermore, we have $||\eta_{v,\us}||_{L^2}=O(\sqrt{|\us|})$.
\end{thm}

The solution to the jump problem $\eta_v$ is constructed explicitly in (\ref{mainjumps}) and (\ref{mainsol}). In order to highlight the series construction, we compute the leading term of the $\us$-expansion for $\eta_v^{(k)}$, which in particular gives the linear term of the $\us$-expansion for $\Omega_\us$. Let $l_v^k=(e_1,\ldots, e_k)$ be a path of length $k$ starting at a given vertex $v=v(e_1)$ in the dual graph $\Gamma_C$. Let $\omega_v(z,w)$ be the fundamental normalized bidifferential (see section $3.1$ for details) on $C_v$ and $\beta_{e,e'}:=\hol(\omega_v)(q_e,q_{e'})$. The $\us$-expansion of $\eta_v^{(k)}$ is given as follows (Proposition \ref{expansion}):
\begin{equation*}
\eta_v^{(k)}(z)=(-1)^{k}\sum_{l_v^k}\prod_{i=1}^{k}s_{e_i}\cdot \omega_v(z,q_{e_1})\prod_{j=1}^{k-1}\beta_{-e_j,e_{j+1}}\hol(\Omega)(q_{-e_k})+O(|\us|^{k+1}),
\end{equation*}
where $z\in C_\us$, and the sum is over the set of paths of length $k$ in the dual graph starting at vertex $v$. One can compare this formula with \cite[Proposition~7]{Wol15}, where a complete expansion is given in Schiffer coordinates. We also want to point out that one can in principle compute for any $k$ the explicit $\us$-expansion of $\eta_v^{(k)}$ (and hence for $\Omega_\us$) via our method.

We are grateful to Scott Wolpert for the following remark: the construction provides a local frame for the sheaf of Abelian differentials near a nodal curve. This implies that the push forward of the relative dualizing sheaf is locally free of the expected rank. The locally-freeness of the first and second powers was first shown in \cite{Ma76} and then in \cite{HK13}. Our method can be further generalized to give a local frame of $k$-differential for any positive $k$ near a boundary point in the moduli space, which will not be treated in this paper.

In Section \ref{sec:periodmatrix}, we compute the leading terms in the variational formula for the period of $\Omega_\us$ over any smooth cycle in $C_\us$. Let the residue of $\Omega$ at $q_e$ be denoted $r_e$, so that $r_e=-r_{-e}$. Let $\alpha$ be any oriented loop in $C$ not contained completely in any irreducible component $C_{v_i}$. Let $\{q_{|e_0|},\ldots ,q_{|e_{N-1}|}\}$ be the ordered collection of nodes that $\alpha$ passes through (with possible repetition), so that $q_{-e_{i-1}}$ and $q_{e_i}$ lie on the same component $C_{v_i}$, and let $q_{|e_N|}=q_{|e_0|}$. Let $\alpha_\us$ be a perturbation of $\alpha$ such that its restrictions on each $C_v$ minus the caps at each node glue correctly to give a loop on $C_\us$.

\begin{thm}(=Theorem \ref{per_gen})\label{intro:generalperiod}
The variational formula of the period of $\Omega_\us$ over $\alpha_\us$ is given by:
\begin{equation*}
\int_{\alpha_\us}\Omega_{\us}=\sum_{i=1}^{N}\left(r_{e_i}\ln{|s_{e_i}|}+c_i+l_i\right)+O(|\us|^2),
\end{equation*}
where $c_i$ and $l_i$ are the constant and linear terms in $\us$ respectively, explicitly given as
\begin{equation*}
\begin{split}
&c_i=\lim_{\us\to 0}\left(\int_{z_{-e_{i-1}}^{-1}(\sqrt{|s_{e_{i-1}}|})}^{z_{e_{i}}^{-1}(\sqrt{|s_{e_i}|})}\Omega_{v_i}-\frac{1}{2}(r_{e_{i-1}}\ln{|s_{e_{i-1}}|}+r_{e_i}\ln{|s_{e_i}|})\right),\\
&l_i:=-\sum_{e\in E_{v_i}}s_e\hol(\Omega)(q_{-e})\cdot\sigma_e,
\end{split}
\end{equation*}
where $E_{v_i}$ denotes the set of nodes on the component $C_{v_i}$,  and
\begin{equation*}
\sigma_e:=
\begin{cases}
\lim_{\us\to 0}\left(\int\limits_{q_{-e_{i-1}}}^{z_{e_i}^{-1}(s_{e_i})}\omega_{v(e_i)}(z_{e_i},q_{e_i})+\frac{1}{s_{e_i}}\right) & \mbox{if } e=e_i;\\
\lim_{\us\to 0}\left(\int\limits^{q_{e_{i}}}_{z_{-e_{i-1}}^{-1}(s_{e_i})}\omega_{v(e_i)}(z_{-e_{i-1}},q_{-e_{i-1}})-\frac{1}{s_{e_i}}\right) & \mbox{if } e=-e_{i-1};\\
\int_{q_{-e_{i-1}}}^{q_{e_i}}\omega_{v(e_i)}(z,q_e) & \mbox{otherwise}.
\end{cases}
\end{equation*}
\end{thm}

This theorem in particular shows that $\int_{\alpha_\us}\Omega_{\us}-\sum_{i=1}^{N}r_{e_i}\ln{|s_{e_i}|}$ is a holomorphic function in $\us$. Using the theorem we can compute the variational formula for the degeneration of the period matrices near an arbitrary stable curve. We choose a suitable symplectic basis $\{A_{k,\us},B_{k,\us}\}_{k=1}^g$ of $H_1(C_\us,\ZZ)$ such that the first $m$ $A$-cycle classes generates the span of the homology classes of the seams. We take a normalized basis $\{v_1,\ldots,v_g\}$ of $H^1(C,K_C)$, such that after applying the jump problem construction, the resulting set of differentials $\{v_{k,\us}\}_{k=1}^g$ is a normalized basis of $H^1(C_\us,K_{C_\us})$ (see Section \ref{sec:notation} and \ref{sec:periodmatrix} for a proper definition). By taking $\Omega=v_k$ and $\alpha_\us=B_{h,\us}$, the following corollary gives the leading terms in the variational formula of the period matrix.

\begin{cor}(=Corollary \ref{periodmatrix})\label{intro:periodmatrix}
For any $e$ and $k$, denote $N_{|e|,k}:=\gamma_{|e|}\times B_{k,\us}$ the intersection product. For any fixed $h,k$, the expansion of $\tau_{h,k}(\us)=\int_{B_{h,\us}}v_{k,\us}$ is given by
\begin{equation*}
\begin{split}
\tau_{h,k}(\us)=&\sum_{e\in E_C}\frac{N_{|e|,h}\cdot N_{|e|,k}}{2}\cdot \ln|s_e|\\
&+\lim_{\us\to 0}\sum_{i=1}^N\left(\int_{z_{-e_{i-1}}^{-1}(\sqrt{|s_{e_{i-1}}|})}^{z_{e_{i}}^{-1}(\sqrt{|s_{e_i}|})}v_k -N_{|e_i|,h}N_{|e_i|,k}\ln{|s_{e_i}|}\right)\\
&-\sum_{e\in E_C}s_e\left(\hol(v_k)(q_e)\hol(v_h)(q_{-e})\right)+O(|\us|^2),
\end{split}
\end{equation*}
where $E_C$ is the set of oriented edges of the dual graph of $C$, $\{q_{|e_i|}\}_{i=0}^{N-1}$ is the set of nodes $B_h$ passes through. Furthermore, under our choice of symplectic basis, $N_{|e|,h}\cdot N_{|e|,k}$ is equal to $1$ if $h=k$ and the node $q_{|e|}$ lies on $B_h$ and equals $0$ otherwise.
\end{cor}

In the paper \cite{Tan91} Taniguchi computed the logarithmic term 
$$
\sum_{e\in E_C}\frac{N_{|e|,h}N_{|e|,k}}{2}\ln|s_e|
$$ 
in the variational formula of $\tau_{h,k}(\us)$. In addition to his result, we get the constant and the linear term as above, and in principle one can also expand the error term to any order. The variational formula of the period matrix is useful in various places. For instance it is used by the first author in \cite{Hu} to compute the degeneration of the theta constants near the boundary of $\overline{\mathcal A }_3$, the (compactified) moduli space of p.p.a.v of dimension three, and furthermore via a modular form method to compute the boundary behavior of the strata $\mathcal{H}(4)$ of the Hodge bundle where the Abelian differentials have a quadruple zero.

When restricted to the case where $C$ has only one node, Theorem \ref{thm_intro} and Corollary \ref{intro:periodmatrix} imply the results in \cite{Yam80}. The computation reproving \cite{Yam80} is done in Section \ref{sec:examples}. In order to demonstrate the greater applicability of our method compared to \cite{Yam80}, in Section \ref{sec:examples} we study two other examples which, to the knowledge of the authors, have not been treated in the literature previously. The first one is the so-called banana curve, that is, curves with two irreducible components meeting at two nodes. The second example is the totally degenerate curve, namely a $\PP^1$ with $2g$ marked points glued in pairs. The last example appears to be important in the study of Teichm\"uller curves.

All of the above are only applicable to the case where $\Omega$ is a stable differential. In Section \ref{sec:hop} we apply the jump problem technique to give a new method to smoothen a differential with higher order zeroes and poles at the nodes. The existing terminology of such a smoothing procedure is {``higher order plumbing"}, as opposed to standard plumbing. The higher order plumbing is a crucial ingredient used by Bainbridge, Chen, Gendron, Grushevsky and M\"oller in the paper \cite{BCGGM16} to construct the \textit{incidence variety compactification} (IVC) of strata of differentials with prescribed zeroes and poles. In that paper, they prove that the necessary and sufficient conditions for a stable differentials $(C,\Omega)$ to lie in the boundary of the IVC is the existence of a {\em twisted differential} $\Xi$ and a level function $l$ on the vertices of the dual graph of $C$, with certain compatibility conditions. The obviously harder direction is the sufficiency of the conditions, which essentially requires a construction of a degenerating family of Abelian differentials in the strata to the limit differential $(C,\Omega)$ with the compatible data $(\Xi,l)$. In \cite{BCGGM16}, the authors give two such constructions by using a plumbing argument and a flat geometry argument respectively. Both arguments only give rise to one-parameter degenerating families. 

In Section \ref{sec:hop}, we briefly revisit their main results, and construct via the jump problem approach a degenerating family that is different from the two mentioned above  (Theorem \ref{hop}). The number of parameters in our degenerating family is equal to the number of levels in the level graph $\Gamma_C$ minus $1$. In particular, we reprove the sufficiency of the conditions for a stable differential to lie in the boundary of the IVC.

\subsection{Prior work} The history of studying the degeneration of period matrices using plumbing coordinates traces back to Fay \cite{Fay73} and Yamada \cite{Yam80}. They only study the case $n=1$, that is, when $C$ has only one node. In this case the authors give complete power series expansion in $s$ (the plumbing parameter at the only node) of Abelian differentials and describe the degeneration of period matrices by a variational formula in terms of $s$ and $\ln{s}$. 

For the case where the stable curve $C$ has multiple nodes, Taniguchi \cite{Tan91} discusses the degeneration of the period matrix and computes the logarithmic term in the variational formula. However, the complete variational formula of neither the Abelian differentials nor the period matrix is derived in his papers. A recent paper of Wolpert \cite{Wol15} gives a complete expansion of Abelian differentials and the second order expansion for the period matrix in terms of Schiffer deformations which are deformation of smooth Riemann surfaces. In \cite{LZR13}, Liu, Zhao, and Rao study holomorphic one-forms and the period matrix under small deformations of a smooth Riemann surfaces to give a complete variational formula using the Kuranishi coordinates on the Teichm\"uller space.


\section{Smoothing Riemann Surfaces}\label{sec:notation}

Let $C$ be a stable nodal curve over the complex numbers. In this section we recall the plumbing construction and fix the notation. 


\begin{df}\label{dualgraph} (Dual graph)
The {\em dual graph} $\Gamma_C$ of a stable curve $C$ is a graph where each unoriented edge corresponds to a node of $C$, and each vertex $v$ corresponds to the normalization of an irreducible component $C_v$ of $C$. The edge connecting vertices corresponds to the node between components.
\end{df}

For future convenience we write $E_C$ for the set of oriented edges $e$ of $\Gamma_C$. We will use $-e$ to denote the same edge as $e$ but with the opposite orientation, and $|e|=|-e|$ the corresponding unoriented edge. Namely, $q_{\pm e}$ are the pre-images of the node $q_{|e|}$ in the normalization of $C$. We write $v(e)$ to denote the source of the oriented edge $e$, and write $E_v$ for the set of edges $e$ such that $v=v(e)$, that is, edges pointing out of the vertex $v$. We denote $|E|_C=\{|e|\}_{e_\in E_C}$ the set of unoriented edges. The cardinality of $|E|_C$ is half of the cardinality of $E_C$.


\subsection{Plumbing construction} 
We now recall the local smoothing procedure of a nodal curve $C$ via {\em plumbing}. There are many equivalent versions of the local plumbing procedure, we follow the one used in \cite{Yam80}, \cite{Wol13} and \cite{GKN17}.

\begin{df}\label{plumbing} (Standard plumbing)
Let $q_e, q_{-e}$ be the two preimages of the node $q_{|e|}$ in the normalization of $C$. Let $z_{\pm e}$ be fixed chosen local coordinates near $q_{\pm e}$. Take a sufficiently small $s_e=s_{-e}\in\CC$, we denote $U_{\pm e}=U_{\pm e}^{s_e}:=\{|z_{\pm e}|<\sqrt{|s_e|}\}\subset C$, and denote $\gamma_{\pm e}:=\partial U_{\pm e}$, which we call the {\em seams}. We orient each seam $\gamma_e$ counter-clockwise with respect to $U_e$. The {\em standard plumbing} $C_{s_e}$ of $C$ is 
$$
C_{s_e}:=[C\backslash U_e\sqcup U_{-e}]/(\gamma_e\sim\gamma_{-e}),
$$
where $\gamma_e\sim\gamma_{-e}$ is identified via the diffeomorphism $I_e:\gamma_e\to\gamma_{-e}$ sending $z_e$ to $z_{-e}={s_e}/{z_e}$. We call the identified seam $\gamma_{|e|}$. The holomorphic structure of $C_{s_e}$ is inherited from $C\backslash \overline{U}_e\sqcup\overline{U}_{-e}$.
\end{df}

\begin{notat}\label{plumbingnotat}
\begin{enumerate}

\item Since $s_e=s_{-e}$, we can use the notation $s_{|e|}$ and denote $\us:=\{s_{|e|}\}_{|E|_C}$. In later parts of the paper we will continue to use $s_e$ (instead of $s_{|e|}$) for simplicity.

\item We write $C_\us$ for the global smoothing of $C$ by plumbing every node $q_{|e|}$ with {\em plumbing parameter} $s_e$, so that $C=C_{\underline 0}$. Let $\hC_v:=C_v\backslash \sqcup_{e\in E_v}U_e$, then $\hC_v$ has boundaries $\ugamma:=\{\gamma_e\}_{e\in E_v}$. We use $\tbC_v$ to denote the interior of $\hC_v$, and $\hC_\us$ to denote the disjoint union of $\hC_v$ for all $v$. We have $C_\us=\hC_\us/\{\gamma_e\sim\gamma_{-e}\}_{|E|_C}$.

\item Throughout this paper, in a specified component $C_v$, the subscripted $z_e$ is used to denote the chosen local coordinate near the node $q_e$ for every $e\in E_v$ for the standard plumbing. The non-subscripted notation $z$ is used to denote an arbitary local coordinate of any point in $\tbC_v$.

\item For future convenience, we denote $|\us|:=\max_{|e|\in |E|_C}|s_e|$.
\end{enumerate}
\end{notat}

\begin{rmk}
Let $u=(u_1,\ldots,u_k)$ be some coordinates along the boundary stratum of $\overline\calM_g$ that $C$ lies in. One can think of the boundary stratum as a Cartesian product of moduli of curves with marked points, and $u$ is the combination of some coordinates chosen on each moduli space. It is a standard result in Teichm\"uller theory (see \cite{IT}, and \cite{HK13}) that the set of plumbing parameters $\us$ together with $u$ give local coordinates on $\overline\calM_g$ near $C$. We denote $C_{u,\us}$ a nearby curve of $C$, then $C=C_{u_0,\underline{0}}$ for some $u_0$. Our results depend smoothly on $u$ throughout the paper, and all the bounds we derive in this paper hold for $u$ varying in a small neighborhood of $u_0$, therefore we fix the coordinate $u_0$ for $C$ and consistently write $C_\us=C_{u_0,\us}$. 

\end{rmk}


\subsection{Conditions on residues} 
Our goal is to express the variational formulas for Abelian differentials with at worst simple poles on $C$ in plumbing coordinates. Take a stable differential $\Omega$ on the stable curve $C$ in the boundary of the Deligne-Mumford compactification $\overline{\mathcal{M}}_g$. Denote $\Omega_v$ to be the restriction of $\Omega$ to the irreducible component $C_v$. We require $\Omega_v$ to be a meromorphic differential whose only singularities are simple poles at the nodes of $C_v$. We denote the residue of $\Omega_v$ at $q_e$ to be $r_e$ (possibly zero). We have $r_e=-r_{-e}$ for any $e\in E_C$ by the definition of the extended Hodge bundle $\Omega\overline{\calM}_g$ (see e.g., \cite{HM}).


\subsection{Jump problem} 
Given a stable differential $\Omega$ on the boundary point $C=C_0$, we want to construct a holomorphic differential $\Omega_\us$ on the nearby curve $C_\us$ by an analytic procedure called solving the {\em jump problem}.

\begin{df}\label{jumpproblem} (Jump problem)
The initial data of the jump problem is a collection $\uphi$ of complex-valued continuous 1-forms $\{\phi_e\}$ supported on the seams $\ugamma$ of $\hC_\us$, satisfying the conditions
\begin{equation}\label{jumpcond}
\phi_e=-I_e^*(\phi_{-e}), \qquad \int_{\gamma_e}\phi_e=0, \qquad\forall e\in E_C.
\end{equation}
We call the set $\{\phi_e\}_{e\in E_C}$ {\em jumps}. A {\em solution to the jump problem} is a holomorphic differential $\eta_\us$ on $\hC_\us$ such that it is holomorphic on $\tbC_\us$ and continuous on the boundaries $\ugamma$, satisfying the condition
\begin{equation*}
\eta_\us|_{\gamma_e}-I_e^*(\eta_\us|_{\gamma_{-e}})=-\phi_e\qquad \forall e\in E_C.
\end{equation*}
\end{df}

Note that by letting $\{\phi_e\}_{E_C}$ be the mis-matches $\{\Omega_{v(e)}|_{\gamma_e}-I^*_e(\Omega_{v(-e)}|_{\gamma_{-e}})\}_{E_C}$ of $\Omega$, one can check that they satisfy (\ref{jumpcond}). Therefore $(\Omega+\eta_\us)|_{C_{v(e)}}$ and $(\Omega+\eta_\us)|_{C_{v(-e)}}$ have no jump along $\gamma_e$ at every node $e$, where $\eta_\us$ is the solution to the jump problem with jumps the mis-matches of $\Omega$. We can thus glue them along each seam to obtain a required global differential $\Omega_\us$ on $C_\us$. 
\begin{notat}
For simplicity, we drop the subscript $\us$ in $\eta_\us$ throughout the paper. But it is important to bear in mind that the solution depends on $\us$ as the size of the seams varies with $\us$.
\end{notat}


\subsection{$A$-normalization} \label{subsec:ANorm}

 Note that the solution to the jump problem is never unique: adding any differential on $C_\us$ gives another solution. We need to impose a normalizing condition to ensure the uniqueness of the solution.

On each irreducible component $C_v$ of the nodal curve $C$ we choose and fix a Lagrangian subspace of $H_1(C_v, \ZZ)$, and we also choose and fix a basis of the subspace. In Definition \ref{plumbing}, the plumbed surface $\hC_\us$ is seen to be a subset of $C$. Since the seams (as boundaries of $\hC_\us$) are contractible on each $C_v$ (without boundaries), we know that the classes of the seams $\{[\gamma_{|e|}]\}_{|E|_C}$ together with the union of the basis of the Lagrangian subspaces on the irreducible components span a Lagrangian subspace of $H_1(C_{\us}, \mathbb{Z})$. We can fix this Lagrangian subspace of $H_1(C_{\us}, \mathbb{Z})$ along the plumbing family $\{C_\us\}$. If some $\gamma_{|e|}$ is homologous to zero on $C_\us$, as $s_e$ approaches $0$ the Lagrangian subspace of $H_1(C_{\us}, \mathbb{Z})$ is invariant; if the class of $\gamma_{|e|}$ is non-zero, then the rank of the Lagrangian subspace drops by 1 as the corresponding element in the basis goes to zero. 

We denote this choice of basis as $\{A_{1,\us},\ldots,A_{g,\us}\}$ where the first $m$ cycles $A_{1,\us},\ldots, A_{m,\us}$ generate the subspace spanned by the seams $\{[\gamma_{|e|}]\}_{|E|_C}$. This choice of indexing will be used later in the computation of the period matrices in Section \ref{subsec:permatrix}. 


\begin{df}\label{def:ANor}
A solution to the jump problem is \textit{$A$-normalized} if it has vanishing periods over $A_{1,\us},\ldots,A_{g,\us}$.
\end{df}

Note that this definition only depends on the choice of Lagrangian subspace of $H_1(C_{\us}, \mathbb{Z})$. In particular by our choice of Lagrangian subspace, an $A$-normalized solution $\eta$ must have vanishing periods over the seams: $\int_{\gamma_{|e|}}\eta=0$. 

It is a standard fact (see e.g., \cite{GH}) that any holomorphic $A$-normalized differential is identically zero on a compact Riemann surface without boundaries. Given two $A$-normalized solutions $\eta$ and $\eta'$ on $\hC_\us$ which are both holomorphic by definition, the differential $\eta-\eta'$ has zero jumps on the seams and thus defines a global holomorphic $A$-normalized differential on $C_\us$, which is therefore identically zero. This shows the uniqueness of an $A$-normalized solution.


\section{Variational Formulas for Stable Differentials}\label{sec:mainresult}


In this section we construct the degenerating family $\Omega_\us$ in a plumbing family $C_\us$, and give the variational formula for $\Omega_\us$ in terms of $\us$. As introduced in Section \ref{intro_results}, we plan to construct the solution to the jump problem that matches the jumps of $\Omega_0=\Omega$. In the classical construction, such differentials are obtained by integrating the jumps against the {\em Cauchy kernel} (see the following section) on the whole $C_\us$. In this approach the Cauchy kernel depends on $\us$, and this dependence is implicit and hard to determine.

Alternatively, following \cite{GKN17}, we fix the Cauchy kernels on each irreducible components of the normalization of the limit stable curve $C$. On each component $C_v$ we integrate the jumps $\{\Omega_{v(e)}|_{\gamma_e}-I^*_e(\Omega_{v(-e)}|_{\gamma_{-e}})\}_{e\in E_v}$ against the Cauchy kernel. In this way we obtain a differential in the classical sense on each component $C_v$ which has jumps across the seams. We then restrict it to $\hC_v$, the component minus the ``caps". In this way the original jumps are compensated by the newly-constructed differentials, but these differentials in turn produce new jumps. However, since the $L_\infty$-norms of the newly-constructed differentials along the seams $\gamma_{|e|}$ in local coordinates $z_e$ are controlled in an explicit way by $\us$, the new jumps are also controlled by $\us$. By iterating the process one obtain a sequence of differentials, each term controlled by a higher power of $\us$. This sequence converges to the desired solution to the jump problem.


\subsection{The Cauchy kernels} 

The construction of the $A$-normalized solution to the jump problem is parallel to the construction of the {\em almost real-normalized} solution in \cite{GKN17}, which uses a different normalizing condition, and the solution differential obtained there allows one to control the reality of periods. 

Given a smooth Riemann surface $C'$, the Cauchy kernel is the unique object on $C'\times C'$, satisfying the following properties: 
\begin{enumerate}
\item It is a meromorphic differential of the second kind in $p$ whose only simple poles are at $p=q$ and $p=q_0$ with residue $\pm\frac{1}{2\pi i}$;
\item It is an $A$-normalized differential in $p$: $
\int_{p\in A_i}K_{C'}(p,q)=0$, for $i=1, \ldots, g$ and $\forall q\in C'$.
\end{enumerate}

The Cauchy kernel can be viewed as a multi-valued meromorphic function in $q$ whose only simple pole is at $p=q$. Let $\{A_i,B_i\}$ be a symplectic basis of $H_1(C',\mathbb{Z})$, and let $\{v_i\}$ to be the basis of holomorphic $1$-form dual to the $A$-cycles. The multi-valuedness is precisely as follows (where $q+\gamma$ denotes the value of the kernel at $q$ upon extension around the loop $\gamma$):
$$
K(p,q+A_i)=K(p,q); \qquad K(p,q+B_i)=K(p,q)+v_i(p).
$$

Note that the Cauchy kernel is a section of a line bundle on $C'\times C'$ satisfying the first two normalization conditions above, and therefore it can be written in terms of theta functions and the Abel-Jacobi map (for a reference of the theta function see \cite{Gun76}). We also remark that $K_{C'}$ depends on the choice of the Lagrangian subspace spanned by the $A$ cycles. For completeness below we include the explicit expression for the Cauchy kernel in terms of theta functions: 

\begin{equation*}
K_{C'}(p,q):=\frac{1}{2\pi i}\frac{\partial}{\partial p} \ln\frac{\theta(A(p)-A(q)-Z)}{\theta(A(p)-Z)},
\end{equation*}
where $\theta$ denotes the theta function on the Jacobian of $C'$, $Z$ denotes a general point of the Jacobian, and $A$ denotes the Abel-Jacobi map with some base point $q_0\in Z$. The expression does not depend on the choice of $Z$.

We call $\omega_{C'}(p,q):=2\pi i d_qK_{C'}(p,q)$ the {\em fundamental normalized bidifferential of the second kind} on $C'$ (also known as the {\em Bergman kernel}). Note that the term ``normalized" here means A-normalization. Namely, 
$$
\int_{p\in A_i}\omega_{C'}(p,q)=0, \qquad i=1, \ldots, g.
$$
The fundamental normalized bidifferential has its only pole of second order at $p=q$. It is uniquely determined by its normalization along $A$-cycles, symmetry in the entries, and the bi-residue coefficient along $p=q$. See notes~\cite[Ch.~6]{Ber06} for a review of the Cauchy kernel and fundamental bidifferential. 

When $C$ is a nodal curve with irreducible components $C_v$, we denote by $K_v$ (resp. $\omega_v$) the Cauchy kernel (resp. bidifferential) on each $C_v$. Recall that $z_e$ (or $w_{e}$, when we need to distinguish between two distinct points in the same neighborhood) denotes the local coordinates in some neighborhood $V_e$ of $q_e$ that contains $\overline{U}_e$.  We define a local holomorphic differential $\bK_v\in\Omega^{1,0}(\sqcup_{e\in E_v}V_e\times\sqcup_{e\in E_v}V_e)$, by taking the regular part of $K_v$:
\begin{equation*}
\bK_v(z_e,w_{e'}):=\begin{cases}
   K_v(z_e,w_{e'}) & \mbox{if } e\neq e', \\
   K_v(z_e,w_{e})-\frac{dz_e}{2\pi i(z_e-w_{e})}    & \mbox{if } e=e'.
  \end{cases}
\end{equation*}

Define $\bw_v(z_e,w_{e'})=2\pi id_{w_{e'}}\bK_v(z_e,w_{e'})$, then similarly we have 

\begin{equation*}
\bw_v(z_e,w_{e'})=\begin{cases}
   \omega_v(z_e,w_{e'}) & \mbox{if } e\neq e', \\
   \omega_v(z_e,w_{e})-\frac{dz_edw_e}{(z_e-w_{e})^2}    & \mbox{if } e=e'.
  \end{cases}
\end{equation*}
For future convenience we fix the notation for the coefficients in the expansion of $\bw_v(z_e,w_{e'})$:
\begin{equation}\label{w_exp}
\bw_v(z_e,w_{e'})=:dw_{e'}dz_e\left(\beta^v_{e,e'}+\sum_{i,j\geq 0,i+j>0}\beta^v_{i,j}z_e^iw_{e'}^j\right).
\end{equation}
Clearly we have $\beta^v_{e,e'}=\bw_v(q_e,q_{e'})$. When the context is clear, we drop the superscript $v$ and write simply $\beta_{e,e'}$ instead.


\subsection{Approach to solving the jump problem}

In this section we approach the jump problem directly in order to clarify the appearance of a series expression for local differentials (\ref{mainjumps}) below. 

We first look at the simplest example: $g=0$. On $\PP^1$ the Cauchy kernel is simply $K(z,w)=\frac{dz}{w-z}$. Let $\gamma$ be a Jordan curve bounding a region $R$ on $\PP^1$. Cauchy's integral formula implies that integrating $K(z,w)$ against a differential $f(w)dw$ holomorphic inside the region $R$ along $\gamma$ (negatively oriented with respect to $R$) vanishes when $z$ is in the exterior of $R$; it is equal to $f(z)dz$ when $z\in R$. In other words integrating $f(w)dw$ against the Cauchy kernel defines a differential on $\PP^1\setminus\gamma$ whose jump along $\gamma$ is precise $f(z)dz$. 

When replicating this idea on Riemann surfaces of higher genus, integrating a differential which is holomorphic inside a contractible loop $\gamma$ against the Cauchy kernel produces a holomorphic differential on that Riemann surface whose jump across $\gamma$ is given by the differential. Below $z^+$ is an point outside $\gamma$ and $z^-$ is inside $\gamma$:

$$
\lim_{z^+\to z'\in\gamma}\int_{\gamma}K(z,w)f(w)dw-\lim_{z^-\to z'\in\gamma}\int_{\gamma}K(z,w)f(w)dw=f(z)dz.
$$
This follows directly from Cauchy's integral formula. In~\cite{GKN17}, or in general when integrating against a jump which is not meromorphic, obtaining a result such as above would require the Sokhotski-Plemelj formula (for reference see~\cite{Ro88}).

We would like to explicitly analyze the dependence of the solution to the jump problem on the plumbing parameters. Solving the jump problem on $C_{\us}$ by integrating against the Cauchy kernel on $C_{\us}$, as described above is classical~\cite{Ro88}, but it does not allow one to study the dependence on plumbing parameters. Therefore the approach we take, which was introduced in~\cite{GKN17}, is integration against fixed Cauchy kernels defined individually on each irreducible component of the nodal curve, which are thus independent of $\us$. The result will give an explicit expansion of the solution in $\us$, and constructing this solution is more involved.

The procedure of the construction of the solution is clarified below. 

{\bf Step 1.} We denote the holomorphic part of the differential $\Omega_{v(e)}(z_e)$ as $\xi_{e}^{(0)}(z_e):=\Omega_v(z_e)-\frac{r_edz_e}{z_e}$. 
It follows from the residue condition that the singular parts of the differentials on the opposite sides of the node cancel. Thus the jumps can be written as follows:
$$
\{\Omega_{v(e)}|_{\gamma_e}-I^*_e(\Omega_{v(-e)}|_{\gamma_{-e}})\}_{e\in E_v}=\{\xi_{e}^{(0)}|_{\gamma_e}-I^*_e(\xi_{-e}^{(0)}|_{\gamma_{-e}})\}_{e\in E_v}\;.
$$

{\bf Step 2.} We integrate the jumps against the Cauchy kernel. This integration defines a differential on the open Riemann surface $\tilde{C}_v$, 
\begin{equation*}
\begin{split}
\eta_v^{(1)}(z):=&\sum_{e\in E_v}\int_{z_e\in\gamma_e}K_v(z,z_e)(\xi_{e}^{(0)}|_{\gamma_e}-I^*_e\xi_{-e}^{(0)}|_{\gamma_{-e}})(z_e)\\&=\sum_{e\in E_v}\int_{z_e\in \gamma_e}K_v(z,z_e)I^*_e\xi_{-e}^{(0)}|_{\gamma_{-e}}(z_e)
\end{split}
\end{equation*}
 where the equality follows from Cauchy's integral formula. We also extend $\eta_v^{(1)}(z)$ continuously to the boundary of the plumbing neighborhood. 

We have an important remark here: The differential $\eta_v^{(1)}(z)$ can be seen as our first attempt at solving the jump problem, but it does {\em not} give the solution of the desired jump problem. There is a new jump between $\Omega_{v(e)}+\eta^{(1)}_{v(e)}$ and $\Omega_{v(-e)}+\eta^{(1)}_{v(-e)}$ on each node. The ``error'' comes from the holomorphic part of the Cauchy kernel. 

{\bf Step 3.} We look at this ``error'' explicitly. Locally near the seam $\gamma_{e_0}$, the differential $\eta_v^{(1)}(z_{e_0})$ for $\sqrt{|s_{e_0}|}<|z_{e_0}|<1$ has the following expression:
\begin{equation*}
\begin{split}
\sum_{{e}\in E_v}\int_{z_{e}\in \gamma_{e}}K_v(z_{e_0},z_{e})I^*_{e}\xi_{-e}^{(0)}(z_{e})=&\sum_{e\in E_v}\int_{z_{e}\in\gamma_{e}}\bK_v(z_{e_0},z_{e})I^*_{e}\xi_{-e}^{(0)}(z_{e})+\\&\frac{1}{2\pi i}\int_{w_{e_0}\in\gamma_{e_0}}\frac{dz_{e_0}}{z_{e_0}-w_{e_0}}I^*_{e_0}\xi_{-e_0}^{(0)}(w_{e_0}).
\end{split}
\end{equation*}
Where we recall that $\bK_v$ is the holomorphic part of $K_v$, and the last part involves the integral of the singular part of the Cauchy kernel. The last integral can be evaluated by Cauchy's integral formula by noting that $|s_ez_{e}^{-1}|<\sqrt{s_e}$ where we point out that $I^*$ is orientation reversing, 
\begin{equation}\label{boundarycomp}
\frac{1}{2\pi i}dz_e\int_{w_e\in\gamma_e}\frac{I^*_{e}\xi_{-e}^{(0)}(w_{e})}{z_e-w_e}=-\frac{1}{2\pi i}\frac{dz_{e}}{z_{e}}\int_{w_{-e}\in\gamma_{-e}}\frac{w_{-e}\xi_{-e}^{(0)}(w_{-e})}{w_{-e}-s_ez_{e}^{-1}}=-s_e\frac{dz_{e}}{z_e^2}\tilde{\xi}_{-e}^{(0)}(\frac{s_e}{z_{e}})=I^*\xi_{-e}^{(0)}(z_e)\;.
\end{equation}
Therefore we have the following:
\begin{equation*}
\begin{split}
\{\eta_{v(e)}^{(1)}|_{\gamma_e}-I^*_e\eta^{(1)}_{v(-e)}|_{\gamma_{-e}}\}_{e\in E_v}=&-\{\Omega_{v(e)}|_{\gamma_e}-I^*_e\Omega_{v(-e)}|_{\gamma_{-e}}\}_{e\in E_v}\\&+\Big(\sum_{e'\in E_v}\int_{z_{e'}\in\gamma_{e'}}\bK_{v(e)}(z_e,z_{e'})I^*_{e'}\xi_{-e'}^{(0)}(z_{e'})\\
&-I^*\sum_{e'\in E_{v(-e)}}\int_{z_{e'}\in\gamma_{e'}}\bK_{v(-e)}(z_{-e},z_{e'})I^*_{e'}\xi_{-e'}^{(0)}(z_{e'})\Big)_{e\in E_v}.
\end{split}
\end{equation*}

Thus we see that $\eta_v^{(1)}(z)$ has the desired jump {\em plus} the jump of (local) holomorphic differentials $\sum_{e'\in E_v}\int_{z_{e'}\in\gamma_{e'}}\bK_v(z_e,z_{e'})I^*_{e'}\xi_{-e'}^{(0)}(z_{e'})$, which is exactly the ``error''. 

{\bf Step 4.} We give an estimate on the size of the ``error''. We show in Lemma \ref{series} below that the $L^\infty$-norm of the ``error'' is controlled by the $L^\infty$-norm of the plumbing parameters $|\us|$. We therefore apply the jump problem {\em again} to further reduce the gap. Finally, our approach to solving the jump problem is by integrating against a series constructed from the recursively appearing jump problems. We prove (in Lemma \ref{series}) that the ``errors'' produced in the $k$-th step of the recursion is controlled by the $k$-th power of $|\us|$, and we use this to show the convergence of the desired solution of the jump problem.

In the following section we first define the ``errors'' $\xi_{e}^{(k)}(z_e)$ in each step, and then we prove Lemma \ref{series} which bounds each by a power of the plumbing parameters, thus the series defined by adding the ``errors" converges. And at last we prove that the solution to the jump problem, denoted $\eta_v$, is the result of integrating this series against the Cauchy kernel on each irreducible component.


\subsection{Construction of the A-normalized solution to the jump problem} 

We construct the A-normalized solution to the jump problem as suggested by the computation above, namely we define local holomorphic differentials, which can be understood as the recursively appearing jumps, and show the series converges. The resulting local differentials are such that when integrated against the Cauchy kernel on each irreducible component of the nodal curve, the jump is given by the first term in the series. 

Let $\Omega$ be a stable differential on the stable curve $C$. We can define recursively the following collection of \textit{holomorphic} differentials described \textit{locally} in the neighborhood of each node :
\begin{equation}\label{mainjumps}
\begin{split}
k=0: \xi_{e}^{(0)}(z_e)&:=\Omega_v(z_e)-\frac{r_edz_e}{z_e};\\
k>0: \xi_{e}^{(k)}(z_e)&:=\sum_{e'\in E_v}\int_{w_{e'}\in\gamma_{e'}}\bK_{v}(z_e,w_{e'})\cdot I^*_{e'}\xi_{-e'}^{(k-1)}(w_{e'}).
\end{split}
\end{equation}

Note $\xi_e^{(k)}$ for $k>0$ depend on $\us$ as $\gamma_e$ and $I^*$ depend on $\us$. We suppress this in the notation. 

Let $\Gamma_C$ be the dual graph of $C$. Let $l^k:=(e_1\ldots,e_{k})$ be an oriented path of length $k$ in the dual graph, starting from the vertex $v=v(e_1)$. We denote $L_v^k$ the collection of all such paths starting from the vertex $v$. We remark that given the Cauchy kernels on each component, the differential $\xi_{e_1}^{(k)}(z_{e_1})$ is determined by the collection of local differentials $\{\xi_{-e_{k}}^{(0)}(w_{-e_k})\}_{L_v^k}$ where $e_k$ is the ending edge of $l^k\in L_v^k$. 

We define $\xi_e(z_e):=\sum_{k=0}^\infty\xi_e^{(k)}(z_e)$. Since $\bK_v(z_e,w_{e'})$ is holomorphic in the first variable, we have $\int_{\gamma_e}\xi_e^{(k)}=0$ for any $e$ and $k$, therefore
\begin{equation}\label{zeroseam_xi}
\int_{\gamma_e}\xi_e=0, \qquad \forall e\in E_C.
\end{equation}
The convergence of this series is ensured by the following lemma, whose proof follows very much along the lines of \cite{GKN17}. We include the proof here for completeness. The essential ingredient of the proof is the fact that our Cauchy kernels and the bidifferentials are defined on the irreducible components, thus they are independent of the plumbing parameters $\us$.

\begin{notat}
\begin{enumerate}
\item Throughout the paper we use the tilde notation to denote the function corresponding to a given differential in a given local coordinate chart, for instance $\bK(z,w)=:\tbK(z,w)dz$, $\bw(z,w)=:\tbw(z,w)dzdw$, and also $\xi_e^{(k)}(z_e):=\txi_e^{(k)}(z_e)dz_e$. 

\item To simplify notation, we denote 
\begin{equation}\label{tildexi}
\txi_{e}:=\txi^{(0)}_e(q_e)
\end{equation}
 at every node $q_e$. 
 
\item When the function $\tw_v(z,w)$ of the bidifferential $\omega_v(z,w)$ is evaluated in the second variable at any node $q_e$, by an abuse of notation, we write $\omega_v(z,q_e)=\tw_v(z,q_e)dz$ for $z\in\tC_v$.

\item Recall that $|s|:=\max_{e\in E_C}|s_e|$. For future convenience, for any collection of holomorphic functions on the unit disks neighborhood at each node $\underline f:=\{f_e\in \mathcal{O}(V_e)\}_{e\in E_C}$, we define the following $L^\infty$-norms:
\begin{equation*}
|f_e|_{\us}:=\sup_{z_e\in\gamma_e}|f_e(z_e)|; \qquad |\underline f|_\us:=\max_{e\in E_C}|f_e|_\us.
\end{equation*}
\end{enumerate}
\end{notat}

Moreover by the Schwarz lemma on $U_e=\{|z_e|<\sqrt{|s_e|}\}$ we have that $|\underline f|_\us\leq|\underline f|_{\underline{1}}\sqrt{|\us|}^{\ord\underline f}$, where $\ord\underline f:=\min_{e\in E_C}\ord_{q_e}f_e$.


\begin{lm}\label{series}
For sufficiently small $\us$, there exists a constant $M_1$ independent of $\us$, such that the following estimate holds:
\begin{equation}\label{xi_bound}
|\underline\txi^{(k)}|_{\us}\leq(|\us|M_1)^k|\underline\txi^{(0)}|_{\us}.
\end{equation}
In particular, the local differential $\xi_e(z_e)$ is a well-defined holomorphic differential at each node $e\in E_C$. 
\end{lm}

\begin{proof}
For all $v$ and all $e, e'\in E_v$, the Cauchy kernel $\tK_v(z_e,w_{e'})$ is analytic in both variables and independent of $\us$. Thus there exists a uniform constant $M_2$ (independent of $\us$) such that for any $z_e\in V_e$,
\begin{equation*}
|\tbK_v(z_e,w_{e'})-\tbK_v(z_e,0)|<M_2|w_{e'}|.
\end{equation*}
This in turn implies 
\begin{equation}\label{K_bound}
\max_{w_{e'}\in\gamma_{e'}}\left|\frac{\tbK(z_e,w_{e'})-\tbK(z_e,0)}{w_{e'}^2}\right|<\frac{M_2}{\sqrt{|s_{e'}|}}.
\end{equation}
By (\ref{zeroseam_xi}), we have
\begin{equation*}
\begin{split}
\left|\int_{w_{e'}\in\gamma_{e'}}\tbK_v(z_e,w_{e'})I^*_{e'}\xi_{-e'}^{(k-1)}({w_{e'}})\right|=\left|\int_{w_{e'}\in\gamma_{e'}}\left[\tbK_v(z_e,w_{e'})-\tbK_v(z_e,0)\right]I^*_{e'}\xi_{-e'}^{(k-1)}({w_{e'}})\right|\\
=|s_{e'}|\int_{w_{e'}\in\gamma_{e'}}\left|\frac{\tbK(z_e,w_{e'})-\tbK(z_e,0)}{w_{e'}^2}\right|\left|I^*_{e'}\txi_{-e'}^{(k-1)}({w_{e'}})\right|dw_{e'}<|s_{e'}|M_2\cdot2\pi|\txi_{-e'}^{(k-1)}|_{\us}
\end{split}
\end{equation*}

The second equality is the result of pulling back $dw_{-e'}$. Note that the last inequality is due to the fact that for $w_{e'}\in\gamma_{e'}$, we have $|I^*_{e'}\txi_{-e'}^{(k-1)}({w_{e'}})|=|\txi_{-e'}^{(k-1)}(\frac{s_{e'}}{w_{e'}})|\leq|\txi^{(k-1)}_{-e'}|_\us$, therefore the integration over $\gamma_{e'}$ gives a ${\sqrt{|s_{e'}|}}$ that cancels the one in (\ref{K_bound}).
By definition of $\txi_e^{(k)}$, there exists a constant $M_1$ independent of $\us$ and $k$ such that,
\begin{equation*}
|\txi_e^{(k)}|_{\us}\leq|\us|M_1\max_{e'\in E_{v(e)}}|\txi_{-e'}^{(k-1)}|_{\us}<|\us|M_1|\underline\txi^{(k-1)}|_{\us}.
\end{equation*}

Note that the RHS is independent of $e$ and $v$, we can thus pass to the maximum over $e\in E_C$ of the LHS and obtain $|\underline\txi^{(k)}|_{\us}<|\us|M_1|\underline\txi^{(k-1)}|_{\us}$. By induction, we have the desired estimate (\ref{xi_bound}).

When $|\us|<2M_1^{-1}$, the geometric series $|\underline\txi|_{\us}:=\sum_{k=0}^\infty|\underline\txi^{(k)}|_{\us}$ converges to a limit bounded by $\left(1+\frac{|\us| M_1}{1-|\us|M_1}\right)|\underline\txi^{(0)}|_{\us}<2|\underline\txi^{(0)}|_{\us}<2\sqrt{|\us|}^{\ord\underline\txi^{(0)}}|\underline\txi^{(0)}|_{\underline 1}$. We therefore conclude that the local differential $\xi_e(z_e)$ is analytic in $\us$.
\end{proof}


We now construct the solution to the jump problem with initial data $\{\Omega_v(e)|_{\gamma_e}-I^*_e(\Omega_{v(-e)}|_{\gamma_{-e}})\}$. We define the following differential on $\hC_v$:
\begin{equation}\label{mainsol1}
\eta_v(z)=\sum_{e\in E_v}\int_{z_e\in \gamma_e}K_v(z,z_e)I^*_e\xi_{-e}(z_e).
\end{equation}
where $z\in \tC_v$. By extending continuously to the seams, the differential $\eta_v$ is defined on $\hC_v$. 

Recall $\xi_e(z_e):=\sum_{k=0}^\infty\xi_e^{(k)}(z_e)$. For future use we denote, 
\begin{equation}\label{mainsol}
\eta_{v}^{(k)}(z):=\sum_{e\in E_v}\int_{z_e\in\gamma_e}K_v(z,z_e)\cdot I^*_e\xi_{-e}^{(k-1)}(z_{e}).
\end{equation}
In this notation we have $\eta_v:=\sum_{k=1}^\infty \eta_v^{(k)}$.

We claim the differentials $\eta_v(z)$ are single-valued. This follows from noticing the multi-valuedness of $K(z,z_e)$ along $B_i$ depends exclusively on $z$, and thus any multi-valuedness is canceled after integration against $I^*_e\xi_{-e}$ by (\ref{zeroseam_xi}).
$$
\int_{z_e\in\gamma_e}K_v(z+B_i,z_e))\cdot I^*_e\xi_{-e}(z_{e})-\int_{z_e\in\gamma_e}K_v(z,z_e))\cdot I^*_e\xi_{-e}(z_{e})=v_i(z)\int_{z_e\in\gamma_e}I^*_e\xi_{-e}(z_{e})=0
$$

Although the Cauchy kernel $K_v$ has a simple pole with residue $(2\pi i)^{-1}$ at the base point $q_0$, it follows from (\ref{zeroseam_xi}) that $\eta_{v}(z)$ is holomorphic at $q_0$ and hence defines a holomorphic differential on $\tbC_v$. Let $\gamma_{q_0}$ be a small loop around the point $q_0$. Here we verify integrating $\eta_v$ along $\gamma_{q_0}$ is zero. The paths $\gamma_{q_0}$ does not intersect any $\gamma_e$, and we could exchange the order of integration in $z$ and $z_e$. The integral of $K_v(z,z_e)$ along $z\in\gamma_{q_0}$ is $(2\pi i)^{-1}$ for any $q$. Thus by (\ref{zeroseam_xi}) integrating the result times $I^*_e\xi_{-e}(z_e)$ along $\gamma_e$ is zero. 


We recall that the $L_2$-norm of a holomorphic differential $\omega$ on a smooth Riemann surface $C'$ is given by $||\omega||_{L_2}:=\frac{i}{2}\int_{C'}\omega\wedge\overline{\omega}$. Note that both $\xi_e(z_e)$ and $\eta_v(z)$ implicitly depend on $\us$. The following theorem establishes an $L^2$ bound on $\eta_v$, and shows that it is the desired solution to the jump problem.


\begin{thm}\label{main}
Let $C$ be a stable nodal curve with irreducible components $C_v$, $\Omega$ a stable differential on $C$. Let $\Omega_v$ be the restriction of $\Omega$ on $C_v$. For $|\us|$ small enough, $\{\eta_{v}\}$ is the unique A-normalized solution to the jump problem with jump data $\Omega_{v(e)}|_{\gamma_e}-I^*_e(\Omega_{v(-e)}|_{\gamma_{-e}})$. 
Moreover, there exists a constant $M$ independent of $v$ and $\us$, such that the following $L^2$-bound of the solution holds:
\begin{equation}\label{sol_bound}
||\eta_v||_{L^2}<\sqrt{|\us|}^{1+\ord\underline\txi^{(0)}}M|\underline\txi^{(0)}|_{\underline 1}.
\end{equation}
Therefore $\{\Omega_{v,\us}:=\Omega_v+\eta_{v}\}$ defines a holomorphic differential when all $|s_e|>0$, denoted $\Omega_\us$ on $C_\us$, satisfying $\Omega_v=\lim_{\us\to 0}\Omega_\us|_{C_v}$ uniformly on compact sets of $C_v\setminus\cup_{e\in E_v}q_e$.
\end{thm}


\begin{proof}

{\bf Step 1.} We first show that the solutions $\eta_v$ are A-normalized. Recall that our choice of the maximal Lagrangian subspace of $H_1(C_\us,\ZZ)$ contains the subspace generated by the classes of the seams $\gamma_{|e|}$. By the fact that $\{K_v(z,z_e)\}_{v,e}$ are A-normalized in the first variable, the integral $\int_{z\in A_{i,\us}}\eta_v^{(k)}(z)$ is zero if the class $[A_{i,\us}]$ does not belong to the span of the seams $\{[\gamma_{|e|}]\}$. 

In order to compute the integrals of $\eta_v^{(k)}$ along the seams, we compute the local expression for $\eta^{(k)}_v$ in the neighborhood $\{|\sqrt{s_e}|<|z_e|<1\}$. Note that the Cauchy kernel $K_v(z_e,w_{e'})$ is holomorphic if $e \neq e'$, and it has a singular part $\frac{dz}{2\pi i(z_e-w_e)}$ when both variables are in the neighborhood of the same nodes. Therefore we have

\begin{equation}\label{SP_fml}
\begin{split}
\eta^{(k)}_v(z_e)&=\frac{dz_e}{2\pi i}\int_{w_e\in\gamma_e}\frac{1}{z_e-w_e}I^*_e\xi^{(k-1)}_{-e}(w_e)+\sum_{e'\in E_v}\int_{w_{e'}\in \gamma_{e'}}\bK_v(z_e,w_{e'})I^*_{e'}\xi_{-e'}(w_{e'}).\\
&=\frac{dz_e}{2\pi i}\int_{w_e\in\gamma_e}\frac{1}{z_e-w_e}I^*_e\xi^{(k-1)}_{-e}(w_e)+\xi^{(k)}_e(z_e)\\
&=\left(I^*_e\xi^{(k-1)}_{-e}+\xi^{(k)}_e\right)(z_e),
\end{split}
\end{equation} 
the equality follows from Cauchy's integral formula, see (\ref{boundarycomp}) for details. Note that $\eta_v^{(k)}$ admits continuous extension to the boundary $\gamma_e$. By this expression and property (\ref{zeroseam_xi}), we conclude that $\int_{z\in\gamma_e}\eta_v^{(k)}(z)=0$ and hence the solution $\eta_v$ is A-normalized.

{\bf Step 2.} We show that the differentials $\{\Omega_{v,\us}\}$ have zero jumps among the seams $\ugamma=\{\gamma_e\}_{e\in E_C}$. It is sufficient to prove
\begin{equation}\label{zerojump}
\left(\Omega_{v(e)}-I^*_e\Omega_{v(-e)}\right)|_{\gamma_e}(z_e)=-\sum_{k=1}^\infty\left(\eta^{(k)}_{v(e)}-I^*_e\eta^{(k)}_{v(-e)}\right)|_{\gamma_e}(z_e).
\end{equation}

First we note that by the opposite residue condition ($r_e=-r_{-e}$) the singular parts of $\Omega_{v(e)}$ and $I^*_e\Omega_{v(-e)}$ cancels, therefore we have $\left(\Omega_{v(e)}-I^*_e\Omega_{v(-e)}\right)|_{\gamma_e}(z_e)=\left(\xi^{(0)}_e-I^*_e\xi^{(0)}_{-e}|_{\gamma_e}\right)(z_e)$. 

For $k\geq1$, by (\ref{SP_fml}) the jumps along the identified seams for each terms $\eta_v^{(k)}$ can be analyzed.
\begin{equation*}
\begin{split}
\left(\eta^{(k)}_v-I^*_e\eta^{(k)}_{v(-e)}\right)(z_e)&=\left(I^*_e\xi^{(k-1)}_{-e}+\xi^{(k)}_e-\xi^{(k-1)}_e-I^*_e\xi^{(k)}_{-e}\right)(z_e)\\
&=\left(\xi^{(k)}_e-I^*_e\xi^{(k)}_{-e}\right)(z_e)-\left(\xi^{(k-1)}_e-I^*_e\xi^{(k-1)}_{-e}\right)(z_e).
\end{split}
\end{equation*}

Therefore $\sum_{k=1}^\infty\left(\eta^{(k)}_{v(e)}-I^*_e\eta^{(k)}_{v(-e)}\right)|_{\gamma_e}(z_e)=-\left(\xi^{(0)}_e-I^*_e\xi^{(0)}_{-e}|_{\gamma_e}\right)(z_e)$, and we have shown (\ref{zerojump}).

{\bf Step 3.} We want to prove the $L^2$-bound (\ref{sol_bound}) for the solution. We take the $L^\infty$ norm of $\teta_v^{(k)}(z_e):=\eta_v^{(k)}(z_e)/dz_e$ on the seams. By (\ref{SP_fml}) we have
\begin{equation}\label{conv}
|\teta_v^{(k)}|_\us:=\max_{|z_e|=|\sqrt{s_e}|}|\teta_v^{(k)}(z_e)|\leq |I^*_e\txi_{-e}^{(k-1)}(z_e)|_\us+|\txi_e^{(k)}|_\us.
\end{equation}

By Lemma \ref{series}, we know that for any $k\geq 1$, there exists a constant $M'$ such that $|\teta_v^{(k)}|_\us<(M'|\us|)^{k-1}|\txi^{(0)}|_\us$. By the summing the series, we have $|\teta_v|_\us<M''|\txi^{(0)}|_\us$ for some constant $M''$.

Take any base point $z_0\in C_v$, define $\pi_v(z):=\int_{z_0}^z\eta_v$. Then since $d\pi_v=\eta_v$, by Stokes theorem, we have 
$$
||\eta_v||^2_{L^2}=\frac{i}{2}\int_{\hat{C}_v}\eta_v\wedge\overline{\eta_v}=\sum_{e\in E_v}\int_{\gamma_e}\overline{\pi}_v\eta_v<M''|\txi^{(0)}|_\us\sum_{e\in E_v}\int_{z_e\gamma_e}|\overline{\pi}_v(z_e)| dz_e.
$$
Since $\eta_v$ is bounded on $\gamma_e$, by taking $z_0\in \gamma_e$ the length of arc from $z_0$ to $z_e\in\gamma_e$ is at most $2\pi\sqrt{|\us|}$. Therefore we can bound $|\overline\pi_v|_\us=|\pi_v|_\us$ by $2\pi\sqrt{|\us|}|\teta_v|_\us=2\pi M''\sqrt{|\us|}|\txi^{(0)}|_\us$. At last we have
$$
||\eta_v||^2_{L^2}<|\us|\cdot(2\pi M''|\txi^{(0)}|_\us)^2\cdot \#E_v.
$$
Thus by letting $M:=2\pi M''\sqrt{\max_v\#E_v}$, since $|\underline\txi^{(0)}|_\us\leq|\underline\txi^{(0)}|_{\underline 1}\sqrt{|\us|}^{\ord\underline\txi^{(0)}}$ we have the required $L^2$-bound (\ref{sol_bound}) for $||\eta_v||_{L^2}$. 

Note that for holomorphic differentials, convergence in $L_2$ sense implies uniform convergence on compact sets. Therefore we conclude that $\Omega_v=\lim_{\us\to 0}\Omega_\us|_{C_v}$ uniformly on compact sets of $C_v\setminus\cup_{e\in E_v}q_e$.

Lastly, the holomorphicity of $\Omega_{v,\us}$ for $\us>0$ follows from the holomorphicity of $\Omega_v$ away from the nodes and the holomorphicity of $\eta_v$ on $\tilde{C}_v$. Recall the Cauchy kernel is holomorphic in $\tilde{C}_v$ except at $q_0$, and we've verified that $\eta_v$ does not have a pole at $q_0$.
\end{proof}

In fact, by construction (\ref{mainsol}) we can compute the $\us$-expansion of each summand $\eta_v^{(k)}$ explicitly, and thus the $\us$-expansion of the differential $\Omega_{\us}$. For future applications and comparisons to earlier works, we compute the first term in the $\us$-expansion for each summand.


\begin{prop}\label{expansion}
Let $l^k=(e_1,\ldots, e_k)\in L_v^k$ be a path of length $k$ in $\Gamma_C$ starting from a given vertex $v=v(e_1)$. Denote $s(l^{k})=\prod_{i=1}^{k}s_{e_i}$, and $\beta(l^k)=\prod_{j=1}^{k-1}\beta_{-e_j,e_{j+1}}$. Then the expansion of $\eta_v^{(k)}$ is given by
\begin{equation}\label{sol_exp}
\eta_v^{(k)}(z)=(-1)^{k}\sum_{l^k\in L_v^{k}}s(l^k)\cdot \omega_v(z,q_{e_1})\beta(l^k)\txi_{-e_k}+O(|\us|^{k+1}),
\end{equation}
where $z\in \hC_v$, $\beta_{e,e'}$ is defined in (\ref{w_exp}), and $\txi_{e}$ is defined in (\ref{tildexi}).
\end{prop}


\begin{proof}
Fix a vertex $v$ in $\Gamma_C$. First for a fixed $e\in E_v$, we show the following expansion for $\xi_e^{(k)}$ for $k>0$:
\begin{equation}\label{indxi}
\xi_e^{(k)}(z_e)=(-1)^{k}\sum_{l^k\in L_v^k}s(l^k)\cdot\bw_v(z_e,q_{e_1})\beta(l^k)\txi_{-e_{k}}+O(|\us|^{k+1}).
\end{equation}

This is derived by induction. For $k=1$, we have $L^1_v=E_v$, and $l^1=(e_1)$ where $e_1\in E_v$. We compute
\begin{equation}\label{xi_first}
\begin{split}
\xi_e^{(1)}(z_e)&=\sum_{e_1\in E_v}\int_{w_{e_1}\in\gamma_{e_1}}\bK_v(z_e,w_{e_1})I_{e_1}^*\xi_{-e_1}^{(0)}({w_{e_1}})\\
&=-\sum_{{e_1}\in E_v}\int_{w_{e_1}\in\gamma_{e_1}}\bK_v(z_e,w_{e_1})\frac{s_{e_1}}{w_{{e_1}}^2}\cdot\txi_{-{e_1}}dw_{e_1}+O(s_{e_1}^2)\\
&=-\sum_{{e_1}\in E_v}s_{e_1}\bw_v(z_e,q_{e_1})\txi_{-{e_1}}+O(|\us|^2),
\end{split}
\end{equation}
where the last equality follows from Cauchy's integral formula. 

For the general case, by applying the inductive assumption (\ref{indxi}) to $I^*_{e_1}\xi^{(k-1)}_{-e_1}$, we have
$$
I^*_{e_1}\xi^{(k-1)}_{-e_1}= (-1)^{k-1}I^*_{e_1}\left(\bw_{v(-{e_1})}(w_{-{e_1}},q_{e_2})\right)\cdot\sum_{l^{k-1}\in L_{v(-e_1)}^{k-1}}s(l^{k-1})\beta(l^{k-1})\txi_{-e_{k}}+O(|\us|^{k}).
$$
Therefore it suffices to prove that for any ${e_1}\in E_{v}$ we have:
\begin{equation}\label{gen_xi}
\int_{w_{e_1}\in\gamma_{e_1}}\bK_v(z_e,w_{e_1})I^*_{e_1}\left(\bw_{v(-{e_1})}(w_{-{e_1}},q_{e_2})\right)=-s_{e_1}\bw_v(z_e,q_{e_1})\beta_{-{e_1},e_2}+O(|\us|^2).
\end{equation}
This is due to $I^*_{e_1}\left(\bw_{v(-{e_1})}(w_{-{e_1}},q_{e_2})\right)=I^*_{e_1}\left((\beta_{-{e_1},e_2}+o(w_{-e_1}))dw_{-e_1}\right)=-\frac{s_{e_1}\beta_{-{e_1},e_2}dw_{e_1}}{w_{e_1}^2}+O(|\us|^2)$ and Cauchy's integral formula. We conclude the induction for (\ref{indxi}).

The expansion (\ref{sol_exp}) for $\eta^{(k)}_v(z)$ is obtained by integrating $\xi^{(k-1)}_{e_1}(z_{e_1})$ against $K_v(z,z_{e_1})$, and the computation is exactly the same as (\ref{gen_xi}):
$$
\int_{w_{e_1}\in\gamma_{e_1}}K_v(z,w_{e_1})I^*_{e_1}\left(\bw_{v(-{e_1})}(w_{-{e_1}},q_{e_2})\right)=-s_{e_1}\omega_v(z,q_{e_1})\beta_{-{e_1},e_2}+O(|\us|^2),
$$
where $z\in\hC_v$. The proof is thus completed.

\end{proof}

\begin{rmk}\label{rmk:linear}
It is important to point out that the expansion (\ref{sol_exp}) is {\em not} the $\us$-expansion of the solution ${\eta}_{v}$, while the latter is also computable by expanding the error term in (\ref{xi_first}) using the higher order coefficients $\beta^v_{ij}$ of $\omega_v$. The explicit formula for the case where $\Gamma_C$ contains only one edge is given by \cite{Yam80}, and will be recomputed (up to the second order) in Section \ref{sec:examples}.

However, as we highlighted by the proposition, it is often more useful and practical to consider $\eta_v$ as the series $\sum_{k=1}^{\infty}\eta_{v}^{(k)}$, given the bound (\ref{conv}) and the recursive construction (\ref{mainsol}). In most cases it is already useful to know the first non-constant term of $\Omega_{v,\us}$, which the proposition suffices to give:
\begin{equation*}
\eta_v^{(1)}(z)=-\sum_{e\in E_v}s_e\omega_v(z,q_e)\txi_{-e}+O(s_e^2).
\end{equation*}
\end{rmk}


\section{Period Matrices in Plumbing Coordinates}\label{sec:periodmatrix}


\subsection{General periods}

Using the construction (\ref{mainsol1}) and expansion (\ref{sol_exp}) of the stable differential $\Omega$, we can compute the variational formula of its periods.

\begin{notat}\label{pmap}
For a stable curve $C$ and its dual graph $\Gamma_C$, define the map $p:H_1(C,\ZZ)\to H_1(\Gamma_C, \ZZ)$ as follows: for the class of a homological (oriented) 1-cycle $[\gamma]$ on $C$, $p([\gamma])$ is the class of the oriented loop in the dual graph that contains the vertices corresponding to the components that $\gamma$ passes, and the edges corresponding to the nodes contained in $\gamma$. The orientation of $p([\gamma])$ is inherited from the orientation of $\gamma$. It is easy to see that the map is surjective, but not injective unless all components have genus zero. Moreover, if $\gamma$ is completely contained in some component $C_v$, then $p([\gamma])=0$. 
\end{notat}

Let $\alpha$ be any closed oriented path on the stable curve $C$, such that $p([\alpha])\neq 0$ (the zero case is trivial in our discussion below). For any small enough $\us$, there exists a small perturbation $\alpha_\us$ of $\alpha$ such that the restriction of $\alpha_\us$ on $\hC_\us$ glues to be a path on $C_\us$. This can be seen by requiring 1) $\alpha_\us\cap\gamma_e=I_e^{-1}(\alpha_\us\cap\gamma_{-e})$ for any seam $\gamma_e$ that $\alpha$ passes; 2) $\alpha_\us$ does not totally contain any seam $\gamma_e$.
By an abuse of notation, the path on $C_\us$ after the gluing is also denoted by $\alpha_\us$.

The following theorem computes the leading terms in the variational formula of $\int_{\alpha_\us}\Omega_\us$. To this end, recall that $U_e=\{|z_e|<\sqrt{|s_e|}\}$ and denote $W_e=\{|z_e|<|s_e|\}$ and $V_e=\{|z_e|<1\}$.


\begin{thm}\label{generalperiod}
For any stable differential $\Omega$ on $C$ with residue $r_e$ at the node $q_e$, let $\alpha$ be any closed oriented path on $C$ such that $p([\alpha])\neq 0$ and $\{e_0,\ldots ,e_{N-1}\}$ be the collection of oriented edges that $p([\alpha])$ passes through (with possible repetition), such that $v(-e_{i-1})=v(e_i)$, and let $e_N=e_0$. Then we have
\begin{equation}\label{per_gen}
\int_{\alpha_\us}\Omega_{\us}=\sum_{i=1}^{N}\left(r_{e_i}\ln{|s_{e_i}|}+c_i+l_i\right)+O(|\us|^2),
\end{equation}
where $c_i$ and $l_i$ are the constant and linear terms in $\us$ respectively, explicitly given as
\begin{equation}\label{per_const}
c_i=\lim_{|\us|\to 0}\left(\int_{z_{-e_{i-1}}^{-1}(\sqrt{|s_{e_{i-1}}|})}^{z_{e_{i}}^{-1}(\sqrt{|s_{e_i}|})}\Omega_v-\frac{1}{2}(r_{e_{i-1}}\ln{|s_{e_{i-1}}|}+r_{e_i}\ln{|s_{e_i}|})\right),
\end{equation}
\begin{equation}\label{per_linear}
l_i:=-\sum_{e\in E_{v(e_i)}}s_e\txi_{-e}\cdot\sigma_e,
\end{equation}
where $\txi_e$ is defined in (\ref{tildexi}), and
\begin{equation}\label{linear_sigma}
\sigma_e:=
\begin{cases}
\lim_{|\us|\to 0}\left(\int\limits_{q_{-e_{i-1}}}^{z_{e_i}^{-1}(s_{e_i})}\omega_{v(e_i)}(z_{e_i},q_{e_i})+\frac{1}{s_{e_i}}\right) & \mbox{if } e=e_i;\\
\lim_{|\us|\to 0}\left(\int\limits^{q_{e_{i}}}_{z_{-e_{i-1}}^{-1}(s_{e_i})}\omega_{v(e_i)}(z_{-e_{i-1}},q_{-e_{i-1}})-\frac{1}{s_{e_i}}\right) & \mbox{if } e=-e_{i-1};\\
\int_{q_{-e_{i-1}}}^{q_{e_i}}\omega_{v(e_i)}(z,q_e) & \mbox{otherwise}.
\end{cases}
\end{equation}
\end{thm}

\begin{rmk}
Prior to the proof of the theorem we have two remarks. Firstly, the period integral in (\ref{per_gen}) depends not only on $p[\alpha]$, but also on the class of the actual path $\alpha$. The integration over $\alpha\cap\hC_v$ gives precisely the constant term (\ref{per_const}). Secondly, note that the limits of the integrals in (\ref{linear_sigma}) are singular because the integrants have a double pole on the nodes. However the singular parts are canceled by $\pm\frac{1}{s_{e_i}}$, so the limits are indeed well-defined. Computations leading to both remarks are contain in the proofs of the following lemma and the theorem.
\end{rmk}

To prove the theorem, it suffices to compute the integral on each component $C_{v(e_i)}, i=1\ldots N$ that $\alpha$ passes through. To simplify notation, throughout the proof below we consider $\alpha$ only passing each component once, while the proof also holds for the general case. Let us denote the intersection of $\alpha_\us$ with $\partial U_{e_i},\partial V_{e_i},\partial U_{-e_{i-1}},\partial V_{-e_{i-1}}$ respectively by $u_{e_i},v_{e_i},u_{-e_{i-1}}$ and $v_{-e_{i-1}}$. Then $\alpha_\us|_{C_{v(e_i)}}$ breaks into three pieces bounded by the four points: 
\begin{enumerate}
\item$\alpha_\us|_{V_{-e_{i-1}}\backslash U_{-e_{i-1}}}$ connecting $u_{-e_{i-1}}$ and $v_{-e_{i-1}}$; 
\item $\alpha_\us|_{\hC_{v(e_i)}\backslash V_{e_i}\cup V_{-e_{i-1}}}$ connecting $v_{e_i}$ and $v_{-e_{i-1}}$;
\item$\alpha_\us|_{V_{e_i}\backslash U_{e_i}}$ connecting $v_{e_i}$ and $u_{e_i}$; 
\end{enumerate}

 For convenience, by a composition of a rotation we can assume that $u_{e_i}=z_{e_i}^{-1}(\sqrt{s_{e_i}})$, $v_{e_i}=z_{e_i}^{-1}(1)$ and similarly $u_{-{e_{i-1}}}=z_{-e_{i-1}}^{-1}(\sqrt{s_{-e_{i-1}}})$, $v_{-{e_{i-1}}}=z_{-e_{i-1}}^{-1}(1)$. Therefore in the lemma and the proofs below, integrations in the local chart $z_{e_i}$ from $u_{e_i}$ to $v_{e_i}$ will be written as from $\sqrt{s_{e_i}}$ to $1$, for any $i=0,\ldots,N-1$. We also remark that this assumption does not change the statement of the theorem.
 

The following lemma simplifies the computation:

\begin{lm}\label{per_trsf}
Given an edge $e$, let $\Omega_{v,\us}$ and $\xi_e^{(k)}$ be defined as before. We have the following equality:
\begin{equation}\label{per_transf}
\begin{split}
\int_1^{\sqrt{s_e}}\Omega_{v(e),\us}(z_e)+\int_{\sqrt{s_e}}^1\Omega_{v(-e),\us}(z_{-e})=& r_e\ln|s_e|+\sum_{k=0}^\infty\int_1^{s_e}\xi_e^{(k)}(z_e)\\
&+\sum_{k=0}^\infty\int_{s_e}^1 \xi_{-e}^{(k)}(z_{-e})
\end{split}
\end{equation}
\end{lm}


\begin{proof}[Proof of Lemma \ref{per_trsf}]
We recall that $\Omega_{v,\us}=\Omega_{v(e_i)}+\sum_k\eta^{(k)}_v$, and as we are concerned with the regular part of the period, locally in the annuli $V_e\setminus W_e$, we have the following expression $\Omega_{v,\us}(z_e)=r_e\frac{dz_e}{z_e}+\xi_e^{(0)}(z_e)+\sum_{k=1}^\infty\eta^{(k)}_v$. The logarithmic term in (\ref{per_transf}) is given by 
$$
\int_1^{\sqrt{s_e}}r_e\frac{dz_e}{z_e}=\frac{1}{2}r_e\ln{|s_e|}
$$
and $r_e=-r_{-e}$. What is left to show is
\begin{equation}\label{lm_initial}
\int_1^{\sqrt{s_e}}\xi_e^{(0)}+\int_{\sqrt{s_e}}^1\xi_{-e}^{(0)}+\sum_{k=1}^\infty\int_1^{\sqrt{s_e}}\eta^{(k)}_v+\sum_{k=1}^\infty\int_{\sqrt{s_e}}^1\eta^{(k)}_{v(-e)}=\sum_{k=0}^\infty\int_1^{s_e}\xi_e^{(k)}+\sum_{k=0}^\infty\int_{s_e}^1 \xi_{-e}^{(k)}
\end{equation}

Note that for each $k\geq 0$, we have $\int_{\sqrt{s_e}}^{s_e}\xi_e^{(k)}(z_e)=\int_{\sqrt{s_e}}^1I_e^*\xi_e^{(k)}(z_{-e})$ and $\int_{s_e}^{\sqrt{s_e}}\xi_{-e}^{(k)}(z_{-e})=\int_1^{\sqrt{s_e}}I_e^*\xi^{(k)}_{-e}(z_{e})$. This gives for $k\geq 0$:
\begin{equation}\label{compensation}
\begin{split}
\int_1^{\sqrt{s_e}}\xi_e^{(k)}(z_e)+\int_{\sqrt{s_{e}}}^1\xi_{-e}^{(k)}(z_{-e})
=&\int_1^{{s_e}}\xi^{(k)}_{e}(z_e)+\int_{{s_e}}^1\xi_{-e}^{(k)}(z_{-e})\\
&-\int_{\sqrt{s_e}}^1I_e^*\xi_{e}^{(k)}(z_{-e})-\int_1^{\sqrt{s_e}}I_e^*\xi_{-e}^{(k)}(z_{e}).
\end{split}
\end{equation}

Grouping the last two terms above with the $(k+1)$ entries in $\sum_{k=1}^\infty\int_1^{\sqrt{s_e}}\eta^{(k)}_v+\sum_{k=1}^\infty\int_{\sqrt{s_e}}^1\eta^{(k)}_{v(-e)}$, and applying (\ref{SP_fml}), we obtain:
\begin{equation}\label{compensation2}
\begin{split}
\int_1^{\sqrt{s_e}}(\eta^{(k+1)}_{v(e)}-I_e^*\xi_{-e}^{(k)})(z_e)+\int_{\sqrt{s_{e}}}^1(\eta^{(k+1)}_{v(-e)}-I_e^*\xi_e^{(k)})(z_{-e})\\
=\int_1^{{\sqrt{s_e}}}\xi^{(k+1)}_{e}(z_e)+\int_{{\sqrt{s_e}}}^1\xi_{-e}^{(k+1)}(z_{-e})
\end{split}
\end{equation}

Summing up both (\ref{compensation}) and (\ref{compensation2}) over all $k\geq 0$ and adding the two equalities together, we immediately obtain (\ref{lm_initial}). The lemma follows.
\end{proof}


\begin{proof}[Proof of Theorem \ref{generalperiod}]

Our goal is to compute the leading terms of the variational formula of $\sum_{i=0}^{N-1}\int_{\alpha_{v(e_i)}}\Omega_{v(e_i),\us}$. To this end, we rearrange the terms and compute the following integrals:
\begin{equation}\label{per_breakdown}
\int_{z^{-1}_{-e_{i-1}}(1)}^{z^{-1}_{e_{i}}(1)}\Omega_{v(e_i),\us}+\int_{1}^{\sqrt{s_{e_i}}}\Omega_{v(e_i),\us}(z_{e_i})+\int_{\sqrt{s_{e_i}}}^{1}\Omega_{v(-e_i),\us}(z_{-e_i})
\end{equation}
It needs to be pointed out that the first two entries above are integrals inside $C_v$, while the last entry is in $C_{v(-e_i)}$. To simplify notation, in the rest of the proof we denote $v:={v(e_i)}$. 

Using the lemma, the last two entries of (\ref{per_breakdown}) are equal to $r_{e_i}\ln|s_{e_i}|+\sum_{k=0}^\infty\int_1^{s_{e_i}}\xi_{e_i}^{(k)}(z_{e_i})+\sum_{k=0}^\infty\int_{s_{e_i}}^1 \xi_{-e_i}^{(k)}(z_{-e_i})$. By definition of $\Omega_{v,\us}$, the first integral in (\ref{per_breakdown}) is equal to $\int_{z^{-1}_{-e_{i-1}}(1)}^{z^{-1}_{e_{i}}(1)}\left(\Omega_{v}+\sum_{k\geq 1}\eta_v^{(k)}\right)$. 

Note that by (\ref{sol_exp}) and (\ref{indxi}), for $k\geq 1$ the integrals of $\xi_{\pm e_i}^{(k)}$ and $\eta_v^{(k)}$ only give terms of order $\geq k$. Also observe that $\int_{v_{-e_{i-1}}}^{v_{e_{i}}}\Omega_v$ is a constant independent of $\us$. Thus to compute the remaining part of the constant term we only have to compute the integrals of $\int_1^{s_{e_i}}\xi_{e_i}^{(0)}(z_{e_i})+\int_{s_{e_i}}^1 \xi_{-e_i}^{(0)}(z_{-e_i})$. 

Since $\xi_{\pm e_i}^{(0)}(z_{\pm e_i})$ is holomorphic in $V_{\pm e_i}$, we have
\begin{equation}\label{lin_omega}
\begin{split}
\int_1^{s_{e_i}}\xi_{e_i}^{(0)}(z_{e_i})+\int_{s_{e_i}}^1 \xi_{-e_i}^{(0)}(z_{-e_i})=&\int_{1}^{0}\xi_{e_i}^{(0)}(z_{e_i})+\int_{0}^1 \xi_{-e_i}^{(0)}(z_{-e_i})\\
&+s_{e_i}\cdot\left(\txi_{e_i}-\txi_{-e_i} \right)+O(|\us|^2).
\end{split}
\end{equation}
Summing up the constant terms on the RHS over $i$, we have computed the constant term (\ref{per_const}).

Now we compute the linear term. Note that $\xi_{v}^{(k)}$ are holomorphic in $\us$. Again by (\ref{indxi}), we only need to compute the integrals of $\xi_{v(\pm e_i)}^{(1)}$, whose expansion is already given by (\ref{xi_first}). Therefore we have the following:
\begin{equation*}\label{lin_xi}
\begin{split}
\int_1^{s_{e_i}}\xi_{e_i}^{(1)}(z_{e_i})=&-\sum_{e\in E_v}s_e\txi_{-e}\int_1^{s_{e_i}}\tbw_{v}(z_{e_i},q_e)dz_{e_i}+O(|\us|^2)\\
=&-\sum_{e\in E_v}s_e\txi_{-e}\int_{1}^{s_{e_i}}\tw_v(z_{e_i},q_e)dz_{e_i}+s_{e_i}\txi_{-e_i}\int_1^{s_{e_i}}\frac{dz_{e_i}}{z_{e_i}^2}+O(|\us|^2)\\
=&s_{e_i}\txi_{-e_i}-\sum_{e\in E_v,e\neq e_i}s_e\txi_{-e}\int_{1}^{0}\tw_v(z_{e_i},q_e)dz_{e_i}\\&-s_{e_i}\txi_{-e_i}\lim_{s_{e_i}\to 0}(\int_1^{s_{e_i}}\tw_v(z_{e_i},q_{e_i})dz_{e_i}+\frac{1}{s_{e_i}})+O(|\us|^2).
\end{split}
\end{equation*}
The existence of the limit above can be seen by integrating the ${1}/{z_{e_i}^2}$ term in the expansion of $\tw_v(z_{e_i},q_{e_i})$.

The linear term in $\int_{s_{e_i}}^1\xi_{-e_i}^{(1)}(z_{-e_i})$ is computed similarly:
\begin{equation*}
\begin{split}
\int_{s_{e_i}}^1\xi_{-e_i}^{(1)}(z_{-e_i})=&-s_{e_i}\txi_{e_i}-\sum_{e\in E_{v(-e_i)},e\neq -e_i}s_e\txi_{-e}\int_{0}^1\tw_v(z_{-e_i},q_e)dz_{-e_i}\\
&-s_{e_i}\txi_{e_i}\lim_{s_{e_i}\to 0}(\int^{1}_{s_{e_i}}\tw_v(z_{-e_i},q_{-e_i})dz_{-e_i}-\frac{1}{s_{e_i}})+O(|\us|^2).
\end{split}
\end{equation*}

Note that the linear terms in (\ref{lin_omega}) have been canceled by the linear terms produced by the singular part of $\omega_v$. Moreover,
\begin{equation*}
\int_{v_{-e_{i-1}}}^{v_{e_i}}\eta_{v}^{(1)}(z)=-\sum_{e\in E_v}s_e\cdot\txi_{-e}\cdot\int_{z^{-1}_{-e_{i-1}}(1)}^{z^{-1}_{e_{i}}(1)}\omega_v(z,q_e)+O(|\us|^2).
\end{equation*}

Summing up all the linear terms above, then summing up over $i$, we have the desired linear term.
\end{proof}

\begin{rmk}\label{comparetotan}
Note that in the proof of the theorem, the function $h(\us):=\int_{\alpha_\us}\Omega_\us-\sum_{i=1}^{N}r_{e_i}\ln{|s_{e_i}|}$ is computed as 
$$
\sum_{i=1}^N\left(\sum_{k=0}^\infty\int_1^{s_{e_i}}\xi_{e_i}^{(k)}(z_{e_i})+\sum_{k=0}^\infty\int_{s_{e_i}}^1 \xi_{-e_i}^{(k)}(z_{-e_i})+\int_{z^{-1}_{-e_{i-1}}(1)}^{z^{-1}_{e_{i}}(1)}\Omega_{v(e_i),\us}\right).
$$
The analyticity of $h(\us)$ in $\us$ follows from the analyticity of each integrand above. The analyticity of $h(\us)$ will be used in our improvement of the result of \cite{Tan91} below. In \cite[Lem.~5.5]{GKN17}, without computing any terms in $h(\us)$, an estimate of $|h(\us)|$ is derived in the real normalized setup.

Moreover, from the proof one can see that besides the complexity of the computation, there is no obstacle in computing every higher order terms in the expansion of the periods of $\Omega_\us$.
\end{rmk}


\subsection{Period matrices}\label{subsec:permatrix}

Recall that in Section \ref{subsec:ANorm} we have chosen a basis $\{A_{i,\us}\}_{i=1}^g$ for a Lagrangian subspace of $H_1(C_\us,\ZZ)$ along the plumbing family. We required that the first $m$ $A$-cycles generate the span of the classes of the seams. In order to study degenerations of the period matrix, we now choose $B_{1,\us},\ldots,B_{g,\us}$ completing the $A$-cycles to a symplectic basis of $H_1(C_{\us},\mathbb{Z})$. The cycles $B_{1,\us},\ldots,B_{g,\us}$ are chosen such that they vary continuously in the family.

One can easily see that for $1\leq k\leq m$, $p([B_{k,\underline 0}])\neq 0$, while for $m+1\leq k\leq g$, the map $p$ annihilates the classes of $B_{k,\underline 0}$. From now on we write $A_k:=A_{k,\underline 0}, B_k:=B_{k,\underline 0}$. Note that for $1\leq k\leq m$, one can also see $B_{k,\us}$ as constructed from $B_k$ by applying a small perturbation as introduced in the previous section. 

 By the Riemann bilinear relations, we define the following basis of Abelian differential $\{v_k\}_{k=1}^g$ in $H^{0}(C,K_C)$ where $C=C_0$ is a stable curve: 
\begin{enumerate}
\item For $m+1\leq k\leq g$, $B_{k}$ is contained in $\tbC_v$ for some $v$, thus $A_{k}$ is contained in the same component. Define $v_k(z)$ to be the Abelian differential dual to $A_{k}$ in $H^{0}(C_v,K_{C_v})$.
\item For $1\leq k\leq m$, $p([B_k])\neq 0$, assume $p([B_{k}])$ passes the edges $e_0,\ldots e_{N-1}$. Define $v_k:=\sum_{i=0}^{N-1}\omega_{q_{e_{i}}-q_{-e_{i-1}}}$, where $\omega_{q_{e_{i}}-q_{-e_{i-1}}}$denotes the $A$-normalized meromorphic differential of the third kind supported on $C_{v(e_i)}$ that has only simple poles at $q_{-e_{i-1}}$ and $q_{e_{i}}$ with residues $-1$ and $1$ correspondingly.
\end{enumerate}

By applying the jump problem construction, we have a collection of Abelian differentials $\{v_{k,\us}\}_{k=1}^g$ for the curve $C_\us$, which is seen to be a normalized basis of $H^{0}(C_\us,K_{C_\us})$. For every $k$ and $|e|$, we have $\int_{\gamma_{|e|}}v_{k,\us}=\int_{\gamma_{|e|}}(v_k+\eta_{k,\us})$ on $C_\us$. Since the solution $\eta_{k,\us}$ to the jump problem with initial jumps of $v_k$ is $A$-normalized, this is equal to the integral $\int_{\gamma_{e}}v_k$ on $C_{v(e)}$. Therefore by the residue theorem, we have $\int_{A_{j,\us}}v_{k,\us}=2\pi i\cdot\delta_{jk}$. This shows that $\{v_{k,\us}\}_{k=1}^g$ is a normalized basis of $H^{0}(C_\us,K_{C_\us})$. The period matrix of $C_\us$ is hence defined to be $\{\tau_{h,k}(\us)\}_{g\times g}$ where $\tau_{h,k}(\us):=\int_{B_{h,\us}}v_{k,\us}$.

We can apply Theorem \ref{generalperiod} to compute the leading terms in the variational formula of $\tau_{h,k}(\us)$. 

\begin{cor}\label{periodmatrix}
For every $|e|$ and $k$, denote $N_{|e|,k}:=\gamma_{|e|}\cdot B_{k,\us}$ the intersection product. For any fixed $h,k$, the expansion of $\tau_{h,k}(\us)$ is given by
\begin{equation}\label{per_matrix}
\begin{split}
\tau_{h,k}(\us)=&\sum_{|e|\in |E|_C}(N_{|e|,h}\cdot N_{|e|,k})\cdot \ln|s_e|\\
&+\lim_{|\us|\to 0}\sum_{i=1}^N\left(\int_{z_{-e_{i-1}}^{-1}(\sqrt{|s_{e_{i-1}}|})}^{z_{e_{i}}^{-1}(\sqrt{|s_{e_i}|})}v_k -N_{|e_i|,h}N_{|e_i|,k}\ln{|s_{e_i}|}\right)\\
&-\sum_{e\in E_C}s_e\left(\hol(\tilde{v}_k)(q_e)\hol(\tilde{v}_h)(q_{-e})\right)+O(|\us|^2),
\end{split}
\end{equation}
where $\{e_i\}_{i=0}^{N-1}$ is the set of oriented edges $p([B_{h}])$ passes through, and $\hol(\tilde{v}_k)$ denotes the regular part of the Laurent expansion of the function of $v_k$ near the nodes of the components where $v_k$ is not identically zero. Furthermore, under our choice of the symplectic basis, $N_{|e|,h}\cdot N_{|e|,k}$ is equal to $1$ if $h=k$ and the node $q_{|e|}$ lies on $B_h$ and equals $0$ otherwise.
\end{cor}

\begin{rmk}
(1) For the purpose of defining the intersection product, we assign an random orientation to $\gamma_{|e|}$. We further remark that there is no canonical way to orient $\gamma_{|e|}$, and the assigned orientation does not affect the statement and the proof.

(2) The main result in \cite{Tan91} is that $h(\us):=\tau_{h,k}(\us)-\sum_{|e|\in |E|_C}(N_{|e|,h}\cdot N_{|e|,k})\cdot \ln|s_e|$ is holomorphic in $\us$. We can see that only the logarithmic term was computed. By Remark \ref{comparetotan}, our corollary in particular reproves his result, and we express more terms in the expansion. 

(3) We want to point out that Taniguchi does not require the classes of ${\gamma_{|e|}}$ to be part of the symplectic basis, therefore $N_{|e|,h}\cdot N_{|e|,k}$ may be any integer. Since the $A,B$-cycles generate $H_1(C_\us,\ZZ)$, the general case follows by linearity.
\end{rmk}

\begin{proof}
We first compute the logarithmic term. Note that the intersection product is independent of $\us$. When $e$ does not lie on $p([B_{h,s}])$, we have $N_{|e|,h}=0$, otherwise $N_{|e|,h}=\pm 1$ and the sign depends on the orientation of $[\gamma_{|e|}]$ compared to that of the corresponding generator $[A_i]$ of the symplectic basis. We now only need to prove that $v_{k}$ has residue $N_{|e|,h}\cdot N_{|e|,k}$ at $q_e$, which is seen as follows: if $e\in p([B_k])$, then by construction of $v_k$, it has residue $\pm 1=N_{|e|,k}$ at $q_{|e|}$ depending again on whether $[\gamma_{|e|}]=[A_i]$ or $-[A_i]$; if $|e|$ does not lie on $p([B_k])$, both the intersection number and the residue are $0$. Note that the signs of $N_{|e|,h}$ and $N_{|e|,k}$ are always the same, therefore $N_{|e|,h}\cdot N_{|e|,k}=\delta_{i,h}\cdot\delta_{i,k}$.

Secondly, we compute the linear term. Note as can be verified from the normalization conditions of the fundamental bidifferential, $v_h(z)=\int_{B_h}\omega(w,z)$ for $m+1\leq h\leq g$, and $v_h(z)=\sum_{i=1}^{N}\int_{q_{-e_{i-1}}}^{q_{e_i}}\omega_{v(e_i)}(w,z)$ for $1\leq h \leq m$, where $\{e_i\}_{i=0}^{N-1}$ is the set of edges $p([B_h])$ passes through. To compute the linear term for $\tau_{h,k}(\us)=\int_{B_{h,\us}}v_{k,\us}$, we observe that by definition of $\sigma_e$ in (\ref{linear_sigma}), we have
$$
\sigma_e=\hol(\tilde v_h)(q_e).
$$
Since $\Omega:=v_k$, we have $\txi_e=\hol(\widetilde\Omega)(q_e)=\hol(\tilde v_k)(q_e)$. Note that $v_h$ is only supported on $\cup_{v\in p([B_h])}C_v$, the sum in (\ref{per_matrix}) is taken over all edges.

Lastly, the constant term follows directly from Theorem \ref{generalperiod}.

\end{proof}



\section{Examples}\label{sec:examples}

In this section we will compute four explicit examples of the variational formula for Abelian differentials and for the period matrix of a stable curve $C$. Throughout this section the stable curve has geometric genus $g$, $\Omega$ is a stable differential on $C$. We choose the symplectic basis of holomorphic 1-cycles and its dual basis of 1-forms as in the previous sections. The notation will vary among the examples according to the structure of $C$.

\subsection{One node}

We first deal with the case where the curve $C$ has only one node $q$. In \cite{Yam80}, Yamada computed the variational formula of both Abelian differentials and the period matrices to any order of the plumbing parameter $s$. We will reprove his result up to the second order, while the full expansion can also be found using our method.

\subsubsection{One node: compact type} When $C$ is of compact type, it has two components $C_1$ and $C_2$ that meet at a single separating node $q$, whose pre-images are denoted by $q_1\in C_1$ and $q_2\in C_2$. Let $z_i$ be the local coordinates near $q_i$. Denote the restriction of $\Omega$ to $C_i$ by $\Omega_i$ ($i=1,2$). The subscripts of the Cauchy kernel and its derivative are changed correspondingly.

Since the curve is of compact type, the differentials $\Omega_i$ have no residue at $q_i$, therefore they are holomorphic and we have $\xi_i^{(0)}(z_i)=\Omega_i(z_i)$. We denote $\txi_i:=\txi_i^{(0)}(q_i)$, and $\Omega_s$ is defined on $C_s$ by formulas (\ref{mainjumps}) (\ref{mainsol}). 
Explicitly, by Proposition \ref{expansion}, the expansion of the restriction $\Omega_{i,s}$ $(i=1,2)$ is given by
\begin{equation}\label{Yam_sep}
\Omega_{i,s}(z)=\Omega_i(z)+\left(-s\cdot \omega_i(z,q_i)\txi_{i'}+s^2\cdot \omega_i(z,q_i)\beta_{i'}\txi_i\right)+O(s^3)
\end{equation}
where by convention, $i'=2$ if $i=1$ and vice versa, and $\beta_{i}$ denotes the leading coefficient in the expansion of $\bw_i$ as in (\ref{w_exp}). 

Let $g_i$ be the genus of $C_i$, with $g_1+g_2=g$. We take a normalized basis of Abelian differentials $\{v_k\}_{k=1}^g$ of $C$ such that $\{v_1,\ldots,v_{g_1}\}$ are supported on $C_1$, and $\{v_{g_1+1},\ldots,v_g\}$ on $C_2$. 

For $1\leq k\leq g_1$, letting $\Omega_1=v_k,\Omega_2=0$ in (\ref{Yam_sep}), we obtain
\begin{equation*}
v_{k,s}(z)=
\begin{cases}
v_k(z)+s^2\cdot \omega_1(z,q_1)\beta_2v_k(q_1)+O(s^3) & z\in\tbC_1,\\
-s\cdot \omega_2(z,q_2)v_k(q_1)+O(s^3) & z\in\tbC_2.
\end{cases}
\end{equation*}
For $g_1+1\leq k\leq g$, the formula is symmetric. We have then reproven \cite[Cor.~1]{Yam80}.

\subsubsection{One node: non-compact type} In this case $C$ is irreducible, with a single node $q$. We denote $q_1,q_2$ the pre-images of $q$ in the normalization $\tC$, and $z_1,z_2$ the corresponding local coordinates. Let $\Omega$ be a meromorphic differential on the normalization of $C$ which has simple poles of residues $r_i$ at $q_i$ $(i=1,2)$. By the residue theorem $r_1=-r_2$.

We now have $\xi_i^{(0)}(z_i)=\Omega(z_i)-\frac{r_idz_i}{z_i}$. Denote $\txi_i=\txi_i^{(0)}(q_i)$. By Proposition \ref{expansion}, we have
\begin{equation}\label{Yam_nonsep}
\Omega_s(z)=\Omega(z)-s\cdot \left(\omega(z,q_1)\txi_2+\omega(z,q_2)\txi_1\right)+O(s^2).
\end{equation}

For a symplectic basis $\{A_{k,s},B_{k,s}\}_{k=1}^g$, we choose $A_{1,s}$ to be the seam, whereas $B_{1,s}$ is taken to intersect $A_{1,s}$ once, oriented from the neighborhood of $q_1$ to the neighborhood of $q_2$. As in Section \ref{subsec:permatrix}, we take $v_1=\omega_{q_2-q_1}$, and $\{v_k\}_{k=2}^{g}$ to be the normalized basis of $H^1(\tC,\CC)$.

By letting $\Omega=v_k$ for $2\leq k\leq g$, we have $r_i=0$, and the equation (\ref{Yam_nonsep}) gives \cite[Cor.~4]{Yam80}. For the case $\Omega=v_1$, we have $r_2=-r_1=1$, then (\ref{Yam_nonsep}) gives \cite[Cor.~5]{Yam80}. Moreover, we can compute the period matrix of $C_s$, reproving \cite[Cor.~6]{Yam80}. By Corollary \ref{periodmatrix} we have
$$
\tau_{1,1}(s)=\int_{B_{1,s}}\omega_{q_2-q_1,s}=\ln |s|+c_{1,1}+s\cdot l_{1,1}+O(s^2),
$$
where $c_{1,1}=\lim_{|s|\to 0}\left(\int_{z_1^{-1}(\sqrt{|s|})}^{z_2^{-1}(\sqrt{|s|})}\omega_{q_2-q_1}-\ln{|s|}\right)$, and $l_{1,1}=-2\sigma_1\sigma_2$.

We also have 
$$
\tau_{k,1}(s)=\tau_{1,k}(s)=\int_{B_{1,s}}v_{k,s}=c_{1,k}+s\cdot l_{1,k}+O(s^2)
$$
 for $2\leq k\leq g$. Since $v_k(x)$ is holomorphic, we have $\xi_i^{(0)}(z_i)=v_k(z_i)$ and hence $\txi_i=v_k(p_i)$. The constant term $c_{1,k}$ is equal to $\int_{q_1}^{q_2}v_k$, and the linear term $l_{1,k}$ is seen to be $-v_k(q_1)\sigma_2-v_k(q_2)\sigma_1$ by (\ref{per_matrix}).

Finally for ${2\leq k,h\leq g}$ we have 
$$
\tau_{k,h}(s)=\tau_{k,h}+s\cdot l_{hk},
$$
where $\tau_{h,k}$ is the period matrix of the normalization of $C$, and $l_{h,k}=-v_k(q_1)v_h(q_2)-v_h(q_1)v_k(q_2)$.


\subsection{Banana curves}

The second example we consider is the stable genus $g$ curve $C$ that has two irreducible components meeting in two distinct nodes (so-called ``banana curve"). This computation has not been done in the literature before.

Let the two components of $C$ be $C_a,C_b$ with genera $g_a$ and $g_b$ where $g_a+g_b=g-1$. The edges corresponding to the two nodes are denoted by $e_1$ and $e_2$. The preimages of the nodes and the local coordinates are denoted as $q_{\pm e_i}$ and $z_{\pm e_i}$ ($i=1,2$), where ``$+$" corresponds to the $C_a$ side, and ``$-$" the $C_b$ side.

 Note that in the case where $C$ has only two components (with any number of nodes connecting them), the path $l^k$ can only go back and forth. Therefore $\xi_{e_i}^{(k)}$ (resp. $\xi_{-e_i}^{(k)}$) ($i=1,2$)  and $\eta_a^{(k)}$ (resp. $\eta_b^{(k)}$) are determined by $\Omega_a$ if $k$ is even (resp. odd), and $\Omega_{b}$ if $k$ is odd (resp. even), as we can see from the terms in expansion (\ref{Yam_sep}). Also note that in expansion (\ref{sol_exp}) of $\eta_a^{(k)}$ and $\eta_b^{(k)}$, there is no residue of $\Omega$ involved. Therefore the we can simplify our computation by assuming that $\Omega_b=0$ and the residues of $\Omega_{a}$ at both nodes are zero. Under these assumptions, we have $\xi_{e_i}^{(0)}(z_{e_i})=\Omega_a(z_{e_i})$ ($i=1,2$). Furthermore, for any integer $k\geq 0$, we have 
 $$\xi_{-e_i}^{(2k)}=\xi_{e_i}^{(2k+1)}=0\qquad(i=1,2),$$
thus by construction (\ref{mainsol}), we have
 \begin{equation*}
 \begin{split}
&\eta^{(2k+1)}_a(z)=0 \qquad z\in\hC_a;\\
&\eta^{(2k)}_b(z)=0\qquad z\in\hC_b.
\end{split}
\end{equation*}
By Proposition \ref{expansion}, we have for $z\in\hC_b$:
\begin{equation*}
\eta_b^{(1)}(z)=-s_1\omega_b(z,q_{-e_1})\txi_{e_1}-s_2\omega_b(z,q_{-e_2})\txi_{e_2}+O(|\us|^2),
\end{equation*}
and for $z\in\hC_a$:
\begin{equation*}
\begin{split}
\eta_a^{(2)}(z)&=s_1^2\omega_a(z,q_{e_1})\beta^b_{1,1}\txi_{e_1}+s_2^2\omega_a(z,q_{e_2})\beta^b_{2,2}\txi_{e_2}\\
&+s_1s_2\left(\omega_a(z,q_{e_1})\beta^b_{1,2}\txi_{e_2}+\omega_a(z,q_{e_2})\beta^b_{2,1}\txi_{e_1}\right)+O(|\us|^3),
\end{split}
\end{equation*}
where $\beta^b_{jk}:=\beta^b_{-{e_j},-{e_k}}$ is the constant term in the expansion of $\omega_b(z_{-e_j},z_{-e_k})$ as in (\ref{w_exp}). 

Note that we can also assume $\Omega_a=0$, and that the residues of $\Omega_b$ at both nodes are zero. The general case follows by adding the differentials in these two cases together.

We now compute the degeneration of period matrix for the banana curve. For the symplectic basis of $H_1(C_\us,\ZZ)$, we let $A_1:=\gamma_{e_2}$, and $B_1$ is taken to intersect each seam once, with the orientation from $q_{e_1}$ to $q_{e_2}$, then from $q_{-e_2}$ to $q_{-e_1}$. Thus we let $v_1:=\omega_{q_{e_2}-q_{e_1}}+\omega_{q_{-e_1}-q_{-e_2}}$ where $\omega_{q_{e_2}-q_{e_1}}$ is supported on $C_a$, and $\omega_{q_{-e_1}-q_{-e_2}}$ on $C_b$. Take $\{A_k,B_k\}_{k=2}^{g_a+1}$ and $\{A_j,B_j\}_{j=g_a+2}^{g}$ to be the symplectic bases of $H_1(C_a,\ZZ)$ and $H_1(C_b,\ZZ)$ respectively. The normalized basis of holomorphic differentials $\{v_k\}_{k=2}^g$ on the two components are taken correspondingly, and we require that $v_k$ is identically zero on $C_b$ if $2\leq k\leq g_a+1$, and on $C_a$ if $g_a+2\leq k\leq g$.

Note that $v_1$ has residues $r_{e_2}=r_{-e_1}=1$, thus we have $\tau_{1,1}(\us)=\ln|s_{1}|+\ln|s_{2}|+c_{1,1}+l_{1,1}+O(|\us|^2)$. By (\ref{per_const}), the constant term is
$$
c_{1,1}=\lim_{|\us|\to 0}\left(\int_{z_{e_1}^{-1}(\sqrt{|s_1|})}^{z_{e_2}^{-1}(\sqrt{|s_{2}|})}\omega_{q_{e_2}-q_{e_1}}+\int_{z_{-e_2}^{-1}(\sqrt{|s_{2}|})}^{z_{-e_1}^{-1}(\sqrt{|s_{2}|})}\omega_{q_{-e_1}-q_{-e_2}}-\ln{|s_{1}|}-\ln{|s_{2}|}\right).
$$
As for the linear term $l_{1,1}$, by (\ref{per_matrix}) we obtain
$$
l_{1,1}=-2s_1\sigma_{-e_1}\sigma_{e_1}-2s_2\sigma_{-e_2}\sigma_{e_2}.
$$

We also see that the expansion of $\tau_{k,1}(\us)=\tau_{1,k}(\us)$ is given by
\begin{equation*}
\tau_{1,k}(\us)=
\begin{cases}
\int_{q_{e_1}}^{q_{e_2}}v_k-s_1v_k(q_{e_1})\sigma_{-e_1}-s_2v_k(q_{e_2})\sigma_{-e_2}+O(|\us|^2) & \mbox{if } 2\leq k\leq g_a+1, \\
\int_{q_{-e_2}}^{q_{-e_1}}v_k-s_2v_k(q_{-e_2})\sigma_{e_2}-s_1v_k(q_{-e_1})\sigma_{e_1}+O(|\us|^2)  & \mbox{if } g_a+2\leq k\leq g.
\end{cases}
\end{equation*}

The remaining $(g-1)\times(g-1)$ minor $\tau_{g-1}(\us):=\{\tau_{h,k}(\us)\}_{h,k=2}^{g}$ of the period matrix is computed as:
\begin{equation*}
\begin{split}
\tau_{g-1}(\us)=&\begin{pmatrix}
\tau_a & 0\\
0 & \tau_b
\end{pmatrix}-s_1\cdot\begin{pmatrix}
0 & ^tR_a(q_{e_1})R_b(q_{-e_1})\\
^tR_b(q_{-e_1})R_a(q_{e_1}) & 0
\end{pmatrix}\\
&-s_2\cdot\begin{pmatrix}
0 & ^tR_a(q_{e_2})R_b(q_{-e_2})\\
^tR_b(q_{-e_2})R_a(q_{e_2})  & 0
\end{pmatrix}+O(|\us|^2),
\end{split}
\end{equation*}
where $\tau_a$ (resp. $\tau_b$) is the period matrix of $C_a$ (resp. $C_b$), and $R_a:=(v_2,\ldots,v_{g_a+1})$, $R_b:=(v_{g_a+2},\ldots,v_g)$.


\subsection{Totally degenerate curves} 

It is a fact that the stable curves that lie in the intersection of the Teichm\"uller curve and the boundary of $\overline\calM_g$ are of arithmetic genus zero. In this subsection we study the largest dimensional boundary stratum of such stable curves and give the variational formula of its period matrix, which to the knowledge of the authors is again not dealt with in literature before. The periods of totally degenerate curves has been studied by Gerritzen in his series of papers \cite{Ger90} \cite{Ger92a} \cite{Ger92b}. The perspectives in those papers are algebraic, mainly by studying the theta functions, the Torreli map and the Schottky problem. No analytic construction such as plumbing is involved.

Let $C$ be a totally degenerate stable curve, namely the normalization $\tC$ is a $\PP^1$ with $g$ pairs of marked points $\{q_{\pm i}\}_{i=1}^g$. Let $q_i$ and $q_{-i}$ be the preimages of the $i$-th node on $C$, and $r_i=-r_{-i}$ be the residue of $\Omega$ at $q_i$.

Let $z$ be the global coordinate, then the local coordinates at the pre-images of nodes are given by
$
z_{\pm i}:=z-q_{\pm i}$, where $i=1,\ldots,g.
$
We have as usual $\xi_{\pm i}^{(0)}(z):=\Omega(z)\mp\frac{r_idz}{z-q_{\pm i}}$, and $\txi_{\pm i}:=\txi^{(0)}_{\pm i}(q_{\pm i})$. 

The Cauchy kernel and the fundamental bidifferential on $\PP^1$ are given explicitly: $K(z,w)=\frac{dz}{2\pi i (z-w)}$; $\omega(z,w)=2\pi i\partial_wK(z,w)=\frac{dzdw}{(z-w)^2}$. We compute $\tw(z,q_i)=\frac{1}{(z-q_i)^2}$, for $i\in\{\pm 1,\ldots,\pm g\}$. The expansion of $\Omega_\us$ is thus given by
\begin{equation*}
\Omega_\us(z)=\Omega(z)-dz\sum_{k=1}^gs_k\left(\frac{\txi_{-k}}{(z-q_k)^2}+\frac{\txi_k}{(z-q_{-k})^2}\right)+O(|\us|^2)
\end{equation*}
where $z\in\hC$.

The classes of the seams $\{[\gamma_i]\}_{i=1}^g$ generate the Lagrangian subgroup of $H_1(C_\us,\ZZ)$, thus we can take $A_i:=\gamma_i$ and $B_i$ the path from $q_{-i}$ to $q_{i}$. The corresponding normalized basis of $1$-forms will be $v_i:=\omega_{q_i-q_{-i}}=dz\left(\frac{1}{z-q_i}-\frac{1}{z-q_{-i}}\right)$ for $i=1,\ldots,g$. Let $\Omega=v_i$, we have for $k\neq\pm i$, $\txi_k=\frac{q_i-q_{-i}}{(q_k-q_i)(q_k-q_{-i})}$ and $\txi_i=\txi_{-i}=\frac{1}{q_{-i}-q_i}$.

One thus computes the period matrix as follows.

\begin{equation*}
\begin{split}
i=j: \tau_{i,i}=&\ln|s_i|-2\ln{|q_i-q_{-i}|}-\frac{2s_i}{(q_i-q_{-i})^2}\\
&-\sum_{k\in\{1,..\hat{i},..g\}}\frac{2s_k(q_i-q_{-i})^2}{(q_k-q_{-i})(q_k-q_i)(q_{-k}-q_{-i})(q_{-k}-q_i)}+O(|\us|^2)
\end{split}
\end{equation*}
\begin{equation*}
\begin{split}
i\neq j: \tau_{i,j}=&\ln{(q_i,q_{-i};q_{j},q_{-j})}-\sum_{k\neq i,j}s_k\Big(\frac{(q_i-q_{-i})(q_j-q_{-j})}{(q_k-q_i)(q_k-q_{-i})(q_{-k}-q_j)(q_{-k}-q_{-j})}\\
&+\frac{(q_i-q_{-i})(q_j-q_{-j})}{(q_k-q_j)(q_k-q_{-j})(q_{-k}-q_i)(q_{-k}-q_{-i})}\Big)\\
&-s_i\frac{q_j-q_{-j}}{q_{-i}-q_i}\Big(\frac{1}{(q_i-q_j)(q_i-q_{-j})}+\frac{1}{(q_{-i}-q_j)(q_{-i}-q_{-j})}\Big)\\
&-s_j\frac{q_i-q_{-i}}{q_{-j}-q_j}\Big(\frac{1}{(q_j-q_i)(q_j-q_{-i})}+\frac{1}{(q_{-j}-q_i)(q_{-j}-q_{-i})}\Big)
+O(|\us|^2)
\end{split}
\end{equation*}
where $(q_i,q_{-i},q_{j},q_{-j})$ stands for the cross-ratio of the (ordered) four points.


\section{Higher Order Plumbing}\label{sec:hop}

In the last section we construct the solution to the jump problem with the initial data arising from the jumps of an Abelian differential that has higher order zeroes and poles at the nodes of the limit curve. Following the terminology in \cite{Gen15} and \cite{BCGGM16}, we call the procedure of smoothing such a differential \textit{higher order plumbing}. We obtain an alternative proof of the sufficiency part of the main theorem in \cite{BCGGM16}. Moreover, our approach gives more information than the two constructions given in that paper. We first give a brief review of definitions and results in \cite{BCGGM16}. Readers that are familiar with this material can safely skip the following subsection.


\subsection{Incidence variety compactification} Let $\mu=(m_1,\ldots,m_n)$ be a partition of $2g-2$, and we assume $m_i>0$ for $i=1,\ldots,n$. We denote $\Omega\calM_{g,n}(\mu)$ to be the stratum whose points are Abelian differentials that have multiplicity $m_k$ at the marked point $p_k$. Let $\PP\Omega\overline\calM_{g,n}^{inc}(\mu)$ be the closure of the strata in the projectivized compactification of the Hodge bundle $\PP\Omega\overline\calM_{g,n}$ over $\overline\calM_{g,n}$, called the {\em incidence variety compactification} (IVC) of the stratum $\Omega\calM_{g,n}(\mu)$ in \cite{BCGGM16}. 

Take a stable pointed differential $(C,\Omega)$ in the boundary of $\Omega\overline\calM_{g,n}$, where $C$ is a stable nodal curve with marked points $p_1,\ldots,p_n$, and $\Omega$ is a stable differential on $C$. Let $(\calC,\mathcal W)\to\Delta$ be a one parameter family in the stratum $\Omega\calM_{g,n}(\mu)$, where $\Delta$ is a disk with parameter $t$, such that $\calC_0=C$. Note that $\Omega$ may be identically zero on some irreducible component $C_v$ of $C$. By an analytic argument \cite[Lemma~4.1]{BCGGM16} one can show that there exist $l_v\in\ZZ_{\leq 0}$ for each $C_v$ such that
$$
\Xi_v:=\lim_{t\to 0}t^{l_v}\Omega_v
$$
is non-zero and not equal to infinity. Such differential $\{\Xi_v\}_v$ must satisfy the following conditions (see the proof of necessity of \cite[Theorem~1.3]{BCGGM16}):
\begin{itemize}
\item [(0)] If $p_k\in C_v$ for some $k$, $\Xi_v$ vanishes to the correct order: $\ord_{p_k}\Xi_v=m_k$;
\item [(1)] The only singularities of $\Xi_v$ are (possible) poles at the nodes of $C_v$; 
\item [(2)] For any node $q_{|e|}$ on $C$, $\ord_{q_e}\Xi_{v(e)}+\ord_{q_{-e}}\Xi_{v(-e)}=-2$;
\item [(3)] If $\ord_{q_e}\Xi_{v(e)}=\ord_{q_{-e}}\Xi_{v(-e)}=-1$ at some node $q_{|e|}$, then the residues are opposite at the node: $\res_{q_e}\Xi_{v(e)}=-\res_{q_{-e}}\Xi_{v(-e)}$.
\end{itemize}
\begin{df}[{\cite[Def.~1.1]{BCGGM16}}]
A differential $\Xi$ satisfying Conditions $(0)\sim (3)$ is called a \textit{twisted differential} of type $\mu$.
\end{df}

Given a one parameter family, $l:v\mapsto l_v$ gives a (full) level function on the vertices of the dual graph $\Gamma_C$. The function $l$ makes $\Gamma_C$ into a level graph, in which the order is denoted by $``\succcurlyeq"$. Moreover, the twisted differential $\Xi$ constructed from the one-parameter family must satisfy the following conditions (again see the proof of necessity of \cite[Theorem~1.3]{BCGGM16}):
\begin{itemize}
\item [(4)] At a node $e$, $v(e)\succcurlyeq v(-e)$ if and only if $\ord_{q_e}\Xi_{v(e)}\geq\ord_{q_{-e}}\Xi_{v(-e)}$, and $v(e)\asymp v(-e)$ if and only if $\ord_{q_e}\Xi_{v(e)}=\ord_{q_{-e}}\Xi_{v(-e)}=-1$
\item [(5)] For any level $L$ in the level graph, for any $v$ such that $l_{v}>L$, let $E^L_v$ be the set of all the nodes $e$ such that $v(e)=v$, $l_{v(-e)}=L$, we have
$$
\sum_{e\in E_v^L}\res_{q_{-e_i}}\Xi_{v(-e_i)}=0.
$$
\end{itemize}
The last condition is called the \textit{Global residue condition} in \cite{BCGGM16}. 

\begin{df}[{\cite[Def.~1.2]{BCGGM16}}]
A twisted differential $\Xi$ is called \textit{compatible} with the stable differential $\Omega$ and the full level function $l$ (or equivalently the full level graph $\Gamma_C$) if (i) $\Xi$ and $l$ satisfy the Conditions (0)$\sim$(5); (ii) the maxima of the level graph correspond to the components $C_v$ where $\Omega_v$ is not identically zero, and on those components, $\Xi_v=\Omega_v$.
\end{df}

The main result of \cite{BCGGM16} is that the necessary and sufficient condition for a pointed stable differential $(C,\Omega)$ to lie in the boundary of the IVC compactification of strata is the existence of a twisted differential $\Xi$ (on $C$) and a full level function $l$ (on $\Gamma_C$) such that $\Xi$ is compatible with $\Omega$ and $l$.


\subsection{Jump problem for higher order plumbing} The proof of sufficiency of this result requires a construction of a family of Abelian differentials in the smooth locus of the strata that degenerates to the limit differential $(C,\Omega)$, given the compatible data $(\Xi,l)$. In \cite{BCGGM16}, the authors give two proofs to the sufficiency by: 1) constructing a one complex parameter family using plumbing; 2) constructing a one real parameter family via a flat geometry argument. We now give a third argument via the jump problem approach. Moreover, the number of parameters over $\CC$ in our degenerating family is equal to the number of levels in $\Gamma_C$ minus $1$. Similar to the plumbing argument used in \cite{BCGGM16}, we will also use a modification differential to match up the residues. Furthermore, the original argument in \cite{BCGGM16} on the operation merging the zeroes will also be applied here to embed the family into the stratum.

Take a plumbing family $\{C_\us\}$ as in Definition \ref{plumbing} such that $C=C_0$, and $\us=\{s_{|e|}\}_{|e|\in |E|_C}$ are the plumbing parameters. Denote the restriction of $\Xi$ on the irreducible component $C_v$ by $\Xi_v$. Let $N_l$ be the number of levels in $\Gamma_C$. Without loss of generality, we assume the range of the level function $l$ to be $\{0,-1,\ldots,1-N_l\}$. 

Assume $j=l_{v(-e)}$, recall that $E_v^j=\{e\in E_v:l_{v(-e)}=j\}$ as defined in condition (5). To glue the twisted differentials $\Xi_v$ and $\Xi_{v(-e)}$, we need to add a {\em modification differential} $\phi_{v,j}$ to $\Xi_v$ in order to match the residue (denoted by $r_{-e}$) of $\Xi_{v(-e)}$. The modification differential $\phi_{v,j}$ is chosen to be any differential which has simple pole at $q_e$ with residue $r_e=-r_{-e}$, where $e\in E^{j}_v$. The global residue condition ensures that the sum of the residues of $\phi_{v,j}$ is zero. The existence of $\phi_{v,j}$ is due to the classical Mittag-Leffler problem. 

For $e\in E_v^j$, assume $\Xi_{v}$ has a zero of order $k_e$ at the node $q_e$, then by Condition $(2)$, $\Xi_{v(-e)}$ has a pole of order $k_e+2$ at the node $q_{-e}$. In order to apply the jump problem to obtain a global differential, the following conditions need to be imposed on the plumbing parameters $\us$:
\begin{itemize}
\item [(i)] For any $e, e'\in E^{j}_v$, we have $s_e^{k_e+1}=s_{e'}^{k_{e'}+1}$;
\item [(ii)] For any two vertices $v_0,v_1$ at different levels (namely $l_{v_0}\neq l_{v_1}$), for any two paths $\{e_i\}_{i\in I}$ and $\{\tilde e_j\}_{j\in J}$ connecting $v_0,v_1$ with $l_{v(e_i)}>l_{v(-e_i)}$ ($\forall i\in I$) and the same for $\{\tilde e_j\}$, we have $\prod_{i\in I}s_{e_i}^{k_{e_i}+1}=\prod_{j\in J}s_{\tilde e_j}^{k_{\tilde e_j}+1}=:s(v_0,v_1)$;
\item [(iii)] If $l_{v_0}=l_{v_1}$, we require that $s(v_0,v_1)=1$.
\end{itemize}

It is important to remark that for such a tuple of plumbing parameters one can deduce that $s(v_0,v_1)$ depends only on the levels of $v_0,v_1$, namely,  $s(v_0,v_1)=s(v_0',v_1')$ as long as $l_{v_0}=l_{v_0'}$ and $l_{v_1}=l_{v_1'}$. We can thus choose one parameter for each level drop:

\begin{df}
Let $t_{i,j}:=s(v_0,v_1)$ where $v_0,v_1$ are two vertices at level $i,j$ respectively. The tuple $\ut:=\{t_{-1},\ldots,t_{1-N_l}\}$ where $t_i:=t_{0,i}/t_{0,i+1}$ are called the {\em scaling parameters}.
\end{df}

Note that $t_{i,j}=\prod_{k=i}^{j}t_k$. The theorem below gives a degenerating family of Abelian differentials parametrized by $\ut$ with central fiber the differential $(C,\Omega)$ in the boundary of the IVC.


\begin{thm}\label{hop}
Let $(C,\Omega,p_1,\ldots,p_n)$ be a stable pointed differential in a given stratum $\Omega\calM_{g,n}(\mu)$. Given the triple $(C,\Xi,l)$ where $\Xi$ is a twisted differential of type $\mu$ on $C$ and $l$ is a full level function on $\Gamma_C$, such that $\Xi$ is compatible with $\Omega$ and $l$, there exists a degenerating family of Abelian differentials $(C_\ut,\Xi_\ut)\subset\Omega\calM_{g,n}(\mu)$ such that $\lim_{\ut\to 0}(C_\ut,\Xi_\ut)=(C,\Omega)$, where $\ut$ are the scaling parameters. 
\end{thm}

\begin{proof}

The proof is completed in three steps. Firstly we construct via the jump problem method a degenerating family of Abelian differentials $(C_\ut,\hXi_\ut)$ in $\Omega\calM_{g,n}$. Then we show that the family lies sufficiently ``near" the stratum, that is, we show that the solution to the jump problem is uniformly controlled by some positive power of $|\ut|:=\max_{1\leq i\leq N_l-1}|t_i|$. Lastly we apply \cite[Lemma~4.7]{BCGGM16} to merge the zeroes of $\hXi_\ut$ to obtain a family contained in the stratum.

For the jump problem construction, we only need to construct the correct initial data $\{\xi_e^{(0)}\}$, the rest of the construction is given by (\ref{mainjumps}) and (\ref{mainsol1}).

Assume that the vertex $v$ lies on the $i$-th level. We define 
$$
\hXi_v:=\Xi_v+\sum_{j<i}t_{i,j}\phi_{v,j},
$$
where $\phi_{v,j}$ is the modification differential we defined earlier.

We now apply the jump problem construction to glue the differentials at the opposite sides of each node $q_{|e|}$. Assume $v(e)$ and $v(-e)$ are on the levels $i$ and $j$ respectively. We glue $t_{0,i}\cdot\hXi_v$ and $t_{0,j}\cdot\hXi_{v(-e)}$ from the opposite sides of the node $q_{|e|}$. Namely, let
\begin{equation*}
\begin{split}
\xi_e^{(0)}(z_e)&:=t_{0,i}\cdot\left(\hXi_v(z_e)-t_{i,j}I^*_eP(\hXi_{v(-e)})(z_e)\right);\\
\xi_{-e}^{(0)}(z_{-e})&:=t_{0,j}\cdot\hol(\hXi_{v(-e)})(z_{-e}),
\end{split}
\end{equation*}
where $P(\cdot)$ denotes the principal part of a differential, and $\hol(\cdot)$ denotes the holomophic part. Conditions $(i)\sim(iii)$ ensures that $t_{0,i}t_{i,j}=t_{0,j}$. 

Note that the initial data $t_{0,i}(\hXi_v-I_e^*\hXi_{v(-e)})(z_e)$ is equal to $(\xi_e^{(0)}-I^*_e\xi_{-e}^{(0)})(z_e)$. In order to apply (\ref{mainjumps}) and (\ref{mainsol1}) to construct the $A$-normalized solution to the jump problem with this initial data, we need to show that $\xi_e^{(0)}(z_e)$ is holomorphic in $z_e$. It is immediate because the pair of differentials $\hXi_v$ and $t_{i,j}\hXi_{v(-e)}$ have opposite residues at the node $q_{|e|}$ and the pull-back of the principal part by $I_e$ is holomorphic.

We recall here the construction of the $A$-normalized solution in (\ref{mainjumps}) and (\ref{mainsol}):  For $k\geq 1$, we define
\begin{equation*}
\begin{split}
&\text{for } z_e\in U_e:\qquad\xi_{e}^{(k)}(z_e):=\sum_{e'\in E_v}\int_{w_{e'}\in\gamma_{e'}}\bK_{v}(z_e,w_{e'})\cdot I^*_{e'}\xi_{-e'}^{(k-1)}(w_{e'});\\
&\text{for } z\in \hC_v:\qquad\eta_{v}^{(k)}(z):=\sum_{e\in E_v}\int_{z_e\in\gamma_e}K_v(z,z_e)\cdot I^*_e\xi_{-e}^{(k-1)}(z_{e}).
\end{split}
\end{equation*}
By Theorem \ref{main}, $\eta_v:=\sum_{k\geq 1}\eta_{v}^{(k)}(z)$ is the $A$-normalized solution to the jump problem of higher order zeroes and poles. 

Similar to the proof of Theorem \ref{main}, we need to show that $\eta_v$ is convergent by giving a $L^2$-bound for the solution $\eta_v$. We can repeat the proof in Lemma \ref{series} and Theorem \ref{main} to get (\ref{sol_bound}), which we recall as
$$
||\eta_v||_{L^2}<\sqrt{|\us|}^{1+\ord\underline\txi^{(0)}}M|\underline\txi^{(0)}|_{\underline 1}
$$
for some constant $M$. We have shown above that $\xi_e^{(0)}(z_e)$ is holomorphic for every $e$, therefore $\ord\underline\txi^{(0)}=\min_e\ord_{q_e}\txi_e^{(0)}\geq 0$. The only thing left to show here is that $|\underline\txi^{(0)}|_{\underline 1}$ is bounded by some power of $\ut$, in other words, the power of $\ut$ in $\xi_{\pm e}^{(0)}$ is non-negative for any $e$.

Note that the power of $\ut$ in $\xi_{-e}^{(0)}$ is automatically non-negative, we only need to show the same holds for $\xi_{e}^{(0)}$. Note that when pulling back the principal part of $\hXi_{v(-e)}$ through $I_e$, its lowest order term $z_{-e}^{-k_e-2}dz_{-e}$ contributes a factor of $s_e^{-k_e-1}$, which is seen to be equal to $t_{i,j}^{-1}$ by condition $(ii)$. Since all other terms in the principal part contribute factors of lower powers of $s_e$, the power of $\ut$ in $\xi_e^{(0)}$ must be non-negative. We can thus apply the same argument as in the proof of Lemma \ref{series} and Theorem \ref{main} and achieve an $L^2$-bound for $\eta_v$.

Let $\hXi_{v,\ut}:=t_{0,i}\hXi_v+\eta_v$ for any $v$ at level $i$, then by the argument in the proof of Theorem \ref{main}, we have that $\{\hXi_{v,\ut}\}_v$ glues to a global differential $\hXi_\ut$ on $C_\ut$ such that $\lim_{\ut\to 0}(C_\ut,\hXi_\ut)=(C,\Omega)$.

Note that by adding the modification differential $\phi_{v,j}$ and the solution differential $\eta_v$ to $\Xi_v$, the zeroes of multiplicity $m_i$ of $\Xi_v$ at $p_i\in C_v$ are broken into $m_i$ simple zeroes in a small neighborhood $U_i$ of $p_i$. The radius of the neighborhood is controlled by the norm of the added differentials. The modification differentials $\phi_{v,j}$ are multiples of $t_{i,j}$, and the argument above gives the $L^2$-bound on $\eta_v$. We can thus merge the zeroes of $\hXi_\ut$ using the arguments in \cite[Lemma~4.7]{BCGGM16} to get the wanted degenerating family $(C,\Xi_\ut)$ with $\ord_{p_i}\Xi_\ut=m_i$, and $\lim_{\ut\to 0}(C_\ut,\Xi_\ut)=(C,\Omega)$.

\end{proof}

Although the differential $\Xi_\ut$ depends on the choice of the modification differentials $\{\phi_v^L\}$,  the existence of such a degenerating family does not rely on the choice of $\{\phi_v^L\}$. Theorem \ref{hop} in particular implies:

\begin{cor}\label{ivc_comp}\cite[Sufficiency Part of Theorem~1.3]{BCGGM16}
A stable pointed differential $(C,\Omega,p_1,\ldots,p_n)$ lies in the boundary of $\PP\Omega\overline\calM_{g,n}^{inc}(\mu)$ if there exist a twisted differential  $\Xi$ of type $\mu$ and a full level function $l$ such that $\Xi$ is compatible with $\Omega$ and $l$.
\end{cor}

{\small \section*{Acknowledgements}
Both authors thank Samuel Grushevsky, the current advisor of the first author and the former advisor of the second author, for the useful discussions throughout the process. The second author would like to thank Marco Bertola and Igor Krichever for conversations related to this work. We thank Scott Wolpert for reading the first version of the paper carefully and his valuable comments. We also thank the referees for their valuable advices which in specific helped us make the paper more accessible. Finally, we thank the organizers of the Cycles on Moduli Spaces, GIT, and Dynamics workshop in ICERM, Brown University, August 2016, where this collaboration started.
}

\end{document}